\documentclass[a4paper,10pt,oneside]{amsart}
\usepackage[utf8]{inputenc}
\usepackage{latexsym, amssymb, amsmath, amsfonts, amsthm, bm, bbm, stmaryrd, thmtools, enumitem, centernot,xspace}
\usepackage{url,graphicx}
\usepackage[rgb,dvipsnames]{xcolor}
\definecolor{svlinks}{rgb}{.0,0.3,0.6}
\usepackage{tikz}
\usetikzlibrary{arrows,arrows.meta,petri,topaths,positioning,shapes,shapes.misc,patterns,calc,decorations.pathreplacing,hobby}
\usepackage[hang]{footmisc}
\PassOptionsToPackage{hyphens}{url}\usepackage[bookmarks,colorlinks=true,pdfhighlight=/O,linkcolor=svlinks,urlcolor=svlinks,citecolor=svlinks,
			pdftitle={Set-Theoretic Blockchains},
			pdfauthor={Miha E. Habič, Joel David Hamkins, Lukas Daniel Klausner, Jonathan L. Verner, Kameryn J. Williams},
			pdfsubject={},
			pdfkeywords={}
			]{hyperref}
\usepackage[margin=3cm]{geometry}
\usepackage{etoolbox}
\pretocmd{\eqref}{Eq.~}{}{}
\usepackage{mathrsfs}

\swapnumbers
\declaretheorem[numberwithin=section]{theorem}
\declaretheorem[numbered=no,name=Theorem]{theorem*}
\declaretheorem[sibling=theorem]{proposition}
\declaretheorem[numbered=no,name=Proposition]{proposition*}
\declaretheorem[sibling=theorem]{corollary}
\declaretheorem[numbered=no,name=Corollary]{corollary*}

\declaretheorem[numbered=no,name=Example]{example*}

\declaretheorem[numbered=no,name=Fact]{fact*}

\declaretheorem[numbered=no,name=Observation]{observation*}
\declaretheorem[sibling=theorem]{lemma}
\declaretheorem[numbered=no,name=Lemma]{lemma*}
\declaretheorem[sibling=theorem]{question}
\declaretheorem[numbered=no,name=Question]{question*}

\swapnumbers

\declaretheorem[numbered=no,name=Claim]{claim*}
\swapnumbers

\theoremstyle{definition}
\declaretheorem[sibling=theorem,name=Definition,qed=$\diamond$]{definition}
\declaretheorem[numbered=no,name=Definition,qed=$\diamond$]{definition*}

\declaretheorem[numbered=no,name=Note]{note*}

\declaretheorem[numbered=no,name=Notation]{notation*}

\newcommand{\Ord}{\mathrm{Ord}}
\DeclareMathOperator{\Add}{Add}
\DeclareMathOperator{\Coll}{Coll}
\newcommand{\graft}{\wr}
\DeclareMathOperator{\dom}{dom}
\newcommand{\pow}{\mathcal{P}}

\renewcommand{\P}{\mathbb{P}}

\newcommand{\rest}[1]{{\upharpoonright}_{#1}} 
\newcommand{\forces}{\Vdash}
\newcommand{\st}{\mid}
\newcommand{\set}[2]{\{#1 \st #2\}}
\newcommand{\seq}[2]{\langle #1 \st #2\rangle}

\newcommand{\ZFC}{\textsc{zfc}\xspace}

\newcommand{\GBC}{\textsc{gbc}\xspace}
\newcommand{\GBCm}{\textsc{gbc}\textsuperscript{$-$}\xspace}
\newcommand{\ETR}{\textsc{etr}\xspace}
\newcommand{\ECC}{\textsc{ecc}\xspace}
\newcommand{\DDG}{\textsc{ddg}\xspace}
\newcommand{\Gdot}{\dot{G}}
\renewcommand{\models}{\vDash}

\setlist{nosep}

\begin{document}
 
\title{Set-Theoretic Blockchains}

\author{Miha E.\ Habi\v{c}}
\address[M.~E.~Habi\v{c}]{
Faculty of Information Technology\\
Czech Technical University in Prague\\
Th\'akurova 9\\
160 00 Praha 6\\
Czech Republic
\&
Department of Logic\\
Faculty of Arts\\
Charles University\\
n\'am.\ Jana Palacha 2\\
116 38 Praha 1\\
Czech Republic
}
\email{habicm@ff.cuni.cz}
\urladdr{https://mhabic.github.io}

\author{Joel David Hamkins}
\address[J.~D.~Hamkins]
{Professor of Logic\\
Faculty of Philosophy\\
University of Oxford\\
Radcliffe Observatory Quarter 555\\
Woodstock Road\\
Oxford\\
OX2 6GG\\
United Kingdom
\&
Sir Peter Strawson Fellow in Philosophy\\
University College\\
Oxford\\
OX1 4BH\\
United Kingdom
}
\email{joeldavid.hamkins@philosophy.ox.ac.uk}
\urladdr{http://jdh.hamkins.org}

\author{Lukas Daniel Klausner}
\address[L.~D.~Klausner]{
Institute of Discrete Mathematics and Geometry\\
TU Wien\\
Wiedner Hauptstra\ss{}e 8--10/104\\
1040 Wien\\
Austria
}
\email{mail@l17r.eu}
\urladdr{https://l17r.eu}

\author{Jonathan Verner}
\address[J.~Verner]{
Department of Logic\\
Faculty of Arts\\
Charles University\\
n\'am. Jana Palacha 2\\
116 38 Praha 1\\
Czech Republic
}
\email{jonathan.verner@ff.cuni.cz}

\DeclareRobustCommand{\okina}{%
	\raisebox{\dimexpr\fontcharht\font`A-\height}{%
		\scalebox{0.8}{`}%
	}%
}
\author{Kameryn J.\ Williams}
\address[K.~J.~Williams]{
University of Hawai\okina{}i at M\=anoa \\
Department of Mathematics \\
2565 McCarthy Mall, Keller 401A \\
Honolulu, HI 96822\\
USA
}
\email{kamerynw@hawaii.edu}
\urladdr{http://kamerynjw.net}

\thanks{We thank the anonymous referee for their helpful comments.\\
The first author was supported by the ESIF, EU Operational Programme Research, Development and Education, the International Mobility of Researchers in CTU project no.~(CZ.02.2.69/0.0/0.0/16\_027/0008465) at the Czech Technical University in Prague, the joint FWF--GA\v{C}R grant no.~17-33849L: Filters, Ultrafilters and Connections with Forcing and by the Progres grant Q14.~Krize racionality a modern\'\i\ my\v{s}len\'\i. \\
The second author is grateful for the financial support provided by the Czech Academy of Sciences and the Winter School in Abstract Analysis 2018 in connection with his visit to Prague and Hejnice in January and February 2018. This work is truly a collaboration between Prague, New York and Vienna, with various combinations of the authors interacting significantly in each of these cities. \\
The third author was supported by the Austrian Science Fund (FWF) project P29575 ``Forcing Methods: Creatures, Products and Iterations''. \\
The fourth author was supported by the joint FWF--GA\v{C}R grant no.~17-33849L: Filters, Ultrafilters and Connections with Forcing, by the Progres grant Q14.~Krize racionality a modern\'\i\ my\v{s}len\'\i\ and by Charles University Research Centre program No.~UNCE/SCI/022.
}

\subjclass[2010]{Primary 03E40; Secondary 03E35}

\keywords{generic multiverse, amalgamability, blockchains, mutual genericity, surgery, exact pair}

\begin{abstract}
	Given a countable model of set theory, we study the structure of its generic multiverse,
	the collection of its forcing extensions and ground models, ordered by inclusion.
	Mostowski showed that any finite poset embeds into the generic multiverse 
	while preserving the nonexistence of upper bounds.
	We obtain several improvements of his result, using what we call the blockchain construction
	to build generic objects with varying degrees of mutual genericity. 
	The method accommodates certain infinite posets, and we can realize these embeddings via a wide 
	variety of forcing notions, while providing	control over lower bounds as well.
	We also give a generalization to class forcing in the context of second-order set theory,
	and exhibit some further structure in the generic multiverse, such as the existence of
	exact pairs.
\end{abstract}

\maketitle

\section{Introduction}
\label{sec:intro}
Forcing is the predominant method of building new models of set theory. From a given model
we may build up myriad forcing extensions, adding new objects to our
models as we proceed; or we may in contrast dig down on a geological tack to various grounds and the mantle, as done recently by Fuchs, the second author, and 
Reitz~\cite{FuchsHamkinsReitz2015:SetTheoreticGeology}, stripping
away any superfluous forcing that might earlier have been performed. By iterating these two operations in turn, we thereby construct a robust collection
of models, all ultimately derived from the initial model and exhibiting a certain family resemblance. Namely, we arrive at the generic multiverse of the original model, first defined by 
Woodin~\cite{Woodin2011:CHGenericMultiverseOmegaConjecture}.

\begin{definition}
	The \emph{generic multiverse} of a countable transitive\footnote{Transitivity is not actually required for
	our arguments, but it simplifies the presentation.} model of set theory $M\models\ZFC$ is the smallest collection of models containing $M$ and 
	closed under forcing extensions and grounds.
\end{definition}

We shall simply refer to \emph{the} generic multiverse, although of course the particular multiverse that arises depends on the original model $M$, or indeed on any of the models in it, for the generic multiverse of a model is an equivalence class in the corresponding partition of the space of all countable transitive models of $\ZFC$. But in this paper, let it be understood that we start from an arbitrary countable transitive model of set theory $M$, fixed for the remainder of the paper, and consider the resulting generic multiverse of this model.

Given a definable class $\Gamma$ of posets, such as \textsc{ccc} or proper posets and so on, one could just as well consider the $\Gamma$-generic multiverse, obtained by forcing extensions and
ground models using only forcing notions in $\Gamma$. The definition of $\Gamma$
should be reinterpreted in each model anew, with the restriction to \textsc{ccc} posets, say, being read \emph{de dicto} rather than \emph{de re}.

Typically, the generic multiverse is of interest because of the models themselves. After all, it arose
because of the study of forcing which gives us new and exciting models of set theory to play with.
But the generic multiverse is also an interesting object \emph{per se} and as a whole. Just as computability
theorists are interested not only in the computational content of particular Turing degrees, 
but also in the global structure of these degrees, so too can set theorists venture beyond the 
particular models in the generic multiverse and study the structure as a whole. The aspect of this structure in which we are most interested in this paper is simply the inclusion relation
between models.

In work growing out of the modal logic of forcing \cite{HamkinsLoewe2008:TheModalLogicOfForcing, HamkinsLoewe2013:MovingUpAndDownInTheGenericMultiverse} and set-theoretic geology \cite{FuchsHamkinsReitz2015:SetTheoreticGeology}, the second author had inquired whether the inclusion relation in the generic multiverse coincides with the forcing-extension relation (see \cite[Question~1]{Hamkins2016:UpwardClosureAndAmalgamationInGenericMultiverse}). Since one might arrive from one model to a smaller model by means of a circuitous zig-zag path of successive forcing extensions and grounds, there seemed at first to be no direct reason to expect the smaller model necessarily to be a ground of the original model. Nevertheless, the second author~\cite[Theorem 3]{Hamkins2016:UpwardClosureAndAmalgamationInGenericMultiverse} had proved that if the downward-directed grounds hypothesis \DDG was true, then indeed the inclusion relation was the forcing-extension relation in the generic multiverse. In a major result, Usuba~\cite{Usuba2017:DownwardDirectedGrounds} subsequently proved that indeed the generic multiverse is downward directed, establishing the following corollary.\footnote{It has been pointed out to us that this lemma may be proved without appealing to \DDG; instead, one can show directly that any pair of models in the generic multiverse enjoys a certain uniform covering property, and then apply a theorem of Bukovsk\'y~\cite{Bukovsky1973:CharacterizationOfGenericExtensions} to see that, if one model is included in another, then the second one is a forcing extensions of the first. We present the given proof since it seems clearer, and \DDG is, after all, a theorem of \ZFC.}

\begin{lemma}
\label{lemma:SubsetEqualsForcingExtension}
	Let $N_0$ and $N_1$ be models in the generic multiverse of $M$. If $N_0\subseteq N_1$ then
	$N_1$ is a forcing extension of $N_0$.
\end{lemma}

\begin{proof}
	The proof amounts to~\cite[Theorem~3]{Hamkins2016:UpwardClosureAndAmalgamationInGenericMultiverse} plus Usuba's result on the \DDG~\cite{Usuba2017:DownwardDirectedGrounds}. Namely, from the \DDG it follows that the collection of forcing extensions of grounds of $M$ is closed under both forcing extensions and grounds, and so this collection of models is precisely the generic multiverse of $M$. If $N_0\subseteq N_1$ are among these, then there is some ground $N$ common to $M$, $N_0$ and $N_1$, 
	which means that $N\subseteq N_0\subseteq N_1$.
	The model $N_0$ is therefore intermediate between a ground model and one of its forcing extensions. It follows by the intermediate model theorem
	(see~\cite[Lemma~15.43 and the discussion preceding Theorem~16.4]{Jech2002:SetTheory}) that
	$N_1$ is also a forcing extension of $N_0$.
\end{proof}

The order structure of the generic multiverse has been studied before. For example, 
Reitz~\cite{Reitz2007:GroundAxiom} studied the ground axiom, a first-order set-theoretic statement
describing the minimal elements in their generic multiverses, and this work subsequently led to the
development of set-theoretic geology (see~\cite{FuchsHamkinsReitz2015:SetTheoreticGeology}),
the study of ground models and ``going down'' in the generic multiverse. The second author and L\"owe had studied the modal logic both of forcing extensions and of grounds \cite{HamkinsLoewe2008:TheModalLogicOfForcing, HamkinsLoewe2013:MovingUpAndDownInTheGenericMultiverse}, and the second author had looked further into issues of upward closure of the generic multiverse in \cite{Hamkins2016:UpwardClosureAndAmalgamationInGenericMultiverse}.

Turning our glance upward, it is clear that the generic multiverse does not have any maximal elements.
However, the upward directedness of the multiverse and the existence of suprema definitely
pose interesting questions.

\begin{definition}
	Let $M$ be a countable model of \ZFC and let $\mathcal{E}$ be a family of forcing extensions
	of $M$. We say that the family $\mathcal{E}$ is \emph{amalgamable} (over $M$) if there
	is a forcing extension $M[G]$ of $M$ that contains every model in $\mathcal{E}$.
\end{definition}

We shall similarly say that a family of generic objects is amalgamable if the corresponding family of generic extensions of $M$ is amalgamable.

It follows from Usuba's results that a collection $\mathcal{E}$ as above is amalgamable precisely when
it has an upper bound in the generic multiverse of $M$. Furthermore, a similar fact holds not just
for forcing extensions of $M$, but throughout the generic multiverse: Any given finite collection of
models in the multiverse has an upper bound if and only if the collection is amalgamable over
some model in the multiverse.

We should also note that the amalgamability of some family of extensions $\set{M[G_i]}{i\in I}$ does
not require the amalgamating model to have the actual set $\set{G_i}{i\in I}$ or the associated
sequence. Of course, this makes no difference in the case of finite $I$, but becomes important
when discussing the (non-)amalgamability of infinite families.

Mostowski~\cite{Mostowski1976:ModelsOfBGAxioms} essentially showed that any finite poset embeds into
the generic multiverse in a way that preserves the nonexistence of upper bounds. 
For example, this means that, given a countable model $M$, there are two forcing extensions
$M[G]$ and $M[H]$ such that no forcing extension of $M$ contains both $M[G]$ and $M[H]$;
in other words, $M[G]$ and $M[H]$ form a nonamalgamable pair.
The extensions constructed by Mostowski are all obtained via Cohen forcing, but further work in \cite{Hamkins2016:UpwardClosureAndAmalgamationInGenericMultiverse} shows 
that the nonamalgamability phenomenon is pervasive and is exhibited by a much larger family of
forcing notions. In the direction of positive structural results, Fuchs, the second author, and 
Reitz~\cite{FuchsHamkinsReitz2015:SetTheoreticGeology} showed that, while not all chains in the
generic multiverse have upper bounds, any countable increasing chain arising from posets of
uniformly bounded size does have an upper bound. Taken together, these results suggest that the
order structure of the generic multiverse is quite complex.

In this paper we will give a number of improvements to Mostowski's result, strengthening
the embedding, generalizing to certain infinite posets, and using forcing notions beyond just
Cohen forcing. All these generalizations are obtained using a powerful method we call
the \emph{blockchain construction}. The method constructs a family of generic objects by
stringing together blocks, each of which conceals a piece of forbidden information.
The blocks are furthermore structured in such a way that this secret remains hidden unless
we have access to a collection of the generics that was meant to be nonamalgamable. In that
case the whole structure of the blocks is revealed, together with the hidden information.

Our main results are the following (see \autoref{sec:Blockchain} 
and \autoref{sec:WideForcing} for definitions):

\begin{theorem}
	If a family of sets $\mathcal{A}$ is defined in $M$ by finite obstacles, then $(\mathcal{A},\subseteq)$
	$*$-embeds into the generic multiverse of $M$.
\end{theorem}

\begin{theorem}
	If a family of sets $\mathcal{A}$ is defined in $M$ by finite obstacles on a set $I$ and $\set{\P_i}{i\in I}$ is a
	family of wide posets in $M$, all the same size and at least as large as $I$, then $(\mathcal{A},\subseteq)$
	$*$-embeds into the generic multiverse, with the additional property that each $A\in\mathcal{A}$ maps to a forcing extension of $M$ by the product $\prod_{i\in A}\P_i$.
\end{theorem}

We also have versions of this theorem for the class-generic multiverse, generated by class forcings over a model of second-order set theory. (See \autoref{sec:ClassWideForcing} for definitions.)

\begin{theorem}
	Let $(M,\mathcal{X})$ be a countable transitive model of \GBC plus the Elementary Class 
	Choice principle. If a family of sets $\mathcal{A}\in M$ is defined in $(M,\mathcal{X})$ by finite obstacles on a class $I\in\mathcal{X}$ and $\set{\P_i}{i \in I}\in \mathcal{X}$ is a family of pretame 
	$\Ord$-wide class forcing notions, then $(\mathcal{A},\subseteq)$ $*$-embeds into
	the class-generic multiverse of $(M,\mathcal{X})$, with each $A\in\mathcal{A}$ mapping to a forcing extension of $(M,\mathcal{X})$ by the product $\prod_{i\in A}\P_i$. 
\end{theorem}

Beyond generalizing Mostowski's result, we also demonstrate some further complexities of the generic
multiverse. For example:

\begin{theorem}
	Suppose $M\subseteq M[c_0]\subseteq M[c_1]\subseteq\dots$ is a countable chain where each $c_n$
	is a Cohen real over $M$. Then this chain admits an exact pair and, consequently, does not have a 
	supremum in the generic multiverse.
\end{theorem}

Most of our results concerning (non-)amalgamability could be seen in the wider context of the \emph{hyperuniverse}, the collection of all countable transitive models of set theory of a fixed height. The study of the hyperuniverse was initiated by Friedman (see e.\,g.~\cite{ArrigoniFriedman2013:HyperuniverseProgram,AntosFriedmanHonzikTernullo2018:HyperuniverseProjectAndMaximality}) and a number of structural properties of both it and the generic multiverse, including forms of some of the results in our paper, were described by Friedman in a tutorial on the topic given in Münster (see~\cite{Friedman2012:HyperuniverseTutorial} for the slides).

\section{The Blockchain Argument---Following Mostowski}
\label{sec:Blockchain}

In this section we will present an adaptation of Mostowski's original argument 
from~\cite{Mostowski1976:ModelsOfBGAxioms} to the context of the generic multiverse, as well as
an improved version thereof. Our purpose in recounting Mostowski's proof is twofold:
first, it will serve as a template for our own proofs 
going forward, and second, the result seems largely unknown and we would like to highlight
it (indeed, we had been unaware of Mostowski's result and found the reference only after having made essentially similar arguments on our own). To be clear,
Mostowski was interested in studying what was substantially a multiverse of countable models
of second-order set theory with a fixed first-order part, together with the corresponding notion of
amalgamability. Nevertheless, the methods apply equally well to the generic multiverse setting (and
we will return to the second-order context in \autoref{sec:ClassWideForcing}).

\begin{theorem}[Mostowski]
\label{thm:Mostowski}
	Let $\ell<\omega$ and let $\mathcal{A}$ be a family of subsets of $\{0,1,\dotsc,\ell\}$ containing all singletons and closed under subsets. 
	Then there are Cohen reals $c_0,c_1,\dotsc,c_\ell$ such that, for any $A\subseteq
	\{0,1,\dotsc,\ell\}$, the family $\set{M[c_n]}{n\in A}$ is amalgamable precisely when
	$A\in\mathcal{A}$.
\end{theorem}

In the proof we will construct the reals $c_n$ step by step, meeting requirements for amalgamability
and nonamalgamability in turn. Since the nonamalgamability steps are the key ones, we first present
a simpler construction, where these are the only steps we need to make.

\begin{proposition}
\label{prop:NonamalgamablePair}
	There are Cohen reals $c,d$ over $M$ such that $M[c]$ and $M[d]$ are nonamalgamable.
\end{proposition}

\begin{definition}
	Let $M$ be a countable transitive model of \ZFC. A real $z$ is \emph{catastrophic} for $M$
	if the height of any model of \ZFC containing $z$ is greater than that of $M$.
\end{definition}

For example, any real coding the height of $M$ will definitely be catastrophic
for $M$. The important feature of catastrophic reals is that they cannot appear in any
model in the generic multiverse of $M$. Consequently, our strategy for ensuring nonamalgamability
will be to code catastrophic reals between various generics.

\begin{proof}[Proof of \autoref{prop:NonamalgamablePair}]
	We will build two descending sequences of Cohen conditions $p_n$ and $q_n$ that will each
	generate their own generic real $c$ and $d$, and along the way we will also make sure
	to split a code of a catastrophic real between $c$ and $d$.
	
	Enumerate the open dense subsets of $\Add(\omega,1)$ in $M$ as $\seq{D_n}{n<\omega}$ and fix
	a catastrophic real $z$ for $M$. Suppose that we have already built conditions
	$p_m,q_m$ for $m<2n$ and each pair of conditions $p_m,q_m$ has the same length. 
	At this stage of the construction we are handed the open dense set $D_n$; we properly extend the
	condition on the $p$-side into $D_n$ and pad the condition on the $q$-side to the
	same length with $0$s. Finally, extend each of the resulting conditions by appending a $1$ and
	then the bit $z(2n)$. This gives us the conditions $p_{2n},q_{2n}$. At the next stage
	of the construction we do the exact same thing but with the roles of the two conditions
	switched and append the bit $z(2n+1)$. We illustrate the construction in 
	\autoref{fig:nonamalgamablePair}.
	
\begin{figure}[h]
	\begin{tikzpicture}[scale=.38,rotate=270]
	\draw[ultra thick] (0,34) -- (0,0) -- (2,0) -- (2,34);
	\draw[dotted,ultra thick] (0,34) -- (0,36) (2,34) -- (2,36); 
	\draw node at (.5,-1.5) {$c\to$};
	\draw node at (1.5,-1.5) {$d\to$};
	\foreach \r/\z [count=\c from 0] in {6/0,11/1,17/0,24/0,32/1} {
		\foreach\i in {0,1} 
		{\draw (\i,\r) rectangle node[scale=.7] {$1$} +(1,1);
			\draw (\i,\r+1) rectangle node[scale=.7] {$\z$} +(1,1);
		}
		\draw (3.5,\r+1.2) node[scale=.6,align=center] {$\begin{aligned}&\uparrow\\z(\c)&=\z\end{aligned}$};}
	\foreach \b/\f/\t in {0/0/5,1/8/10,0/13/16,1/19/23,0/26/31}
	{\foreach \i in \b {
			\foreach \r in {\f,...,\t} {
				\draw[pattern=north east lines] (\i,\r) rectangle +(1,1);
			}
	}}
	\end{tikzpicture}
	\caption{Coding $z$ between the two generics. The clear areas are padded with $0$s and each
	point coding a bit of $z$ is preceded by a start marker---a single '1'.}
	\label{fig:nonamalgamablePair}
\end{figure}
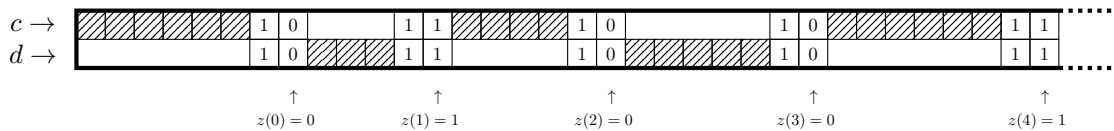

	We now let $c:=\bigcup_n p_n$ and $d:=\bigcup_n q_n$. Clearly both $c$ and $d$
	are Cohen reals over $M$. But note that any model containing both of them would
	be able to recover $z$, since the places where we coded $z$ are precisely those
	places which follow a start marker (i.\,e.\ both $c$ and $d$ have a $1$ on the previous
	place---the start marker---and at least one of $c$ and $d$ has a $0$ immediately 
	preceding the start marker). This means that
	$M[c]$ and $M[d]$ cannot be amalgamated, since $z$ is catastrophic for $M$.
\end{proof}

The key to this argument was that the two generic reals were built in blocks, and within
each block only one of them was permitted to be active, that is, to have nonzero bits.
This allowed us to use the places where both were active to mark the endpoints of the blocks
and thus the coding points where $z$ was hidden.

The proof of Mostowski's theorem will use the same strategy, building
a chain of blocks and only permitting certain columns to be active within each block.

\begin{proof}[Proof of \autoref{thm:Mostowski}]
	Fix an enumeration in order type $\omega$ of all pairs $(A,D)$ where $A\in\mathcal{A}$
	and $D\in M$ is an open dense subset of $\Add(\omega,A)$. Also fix a catastrophic real $z$
	for $M$. We again build descending sequences of conditions $p^n_k$ for $k\leq\ell$, seen as filling in an
	$\omega$-by-$(\ell+1)$ matrix with $0$s and $1$s. The columns of this matrix will grow
	into the desired generic reals $c_k$. We will also
	ensure that at each step of the construction we have completely filled in the matrix
	below some row and that there are no entries above this row.
	
	Suppose that we are at some stage in our construction and we are handed an $A\in\mathcal{A}$
	and a dense set $D\subseteq\Add(\omega,A)$. We now properly extend the columns with indices in $A$
	to collectively meet the dense set $D$ and pad the remaining columns with $0$s to make sure
	that all columns have the same height. Finally, extend each column by appending
	a $1$ and then the next bit of $z$. We illustrate the first few steps of this construction
	in \autoref{fig:setTheoreticBlockchain}.
	
	\begin{figure}[h]
		\begin{tikzpicture}[scale=.35]
		\draw[ultra thick] (0,26) -- (0,0) -- (10,0) -- (10,26);
		\draw[dotted,ultra thick] (0,26) -- (0,28) (10,26) -- (10,28); 
		\foreach \r/\z [count=\c from 0] in {6/0,11/1,17/0,24/0} {
			\foreach\i in {0,...,9} 
			{\draw (\i,\r) rectangle node[scale=.7,opacity=1] {$1$} +(1,1);
				\draw (\i,\r+1) rectangle node[scale=.7,opacity=1] {$\z$} +(1,1);
			}
			\draw (12.5,\r+1.5) node[scale=.7] {$\leftarrow z(\c)=\z$};}
		\foreach \b/\f/\t in {{2,3,4,5}/0/5,{3,4,5,6,7}/8/10,{0,1,2,3,7,8,9}/13/16,{2,3,4,5}/19/23}
		{\foreach \i in \b {
				\foreach \r in {\f,...,\t} {
					\draw[pattern=north east lines] (\i,\r) rectangle +(1,1);
				}
		}}
		\end{tikzpicture}
		\caption{The set-theoretic blockchain of \autoref{thm:Mostowski}. Clear areas are again padded with $0$s.}
	\label{fig:setTheoreticBlockchain}
	\end{figure}
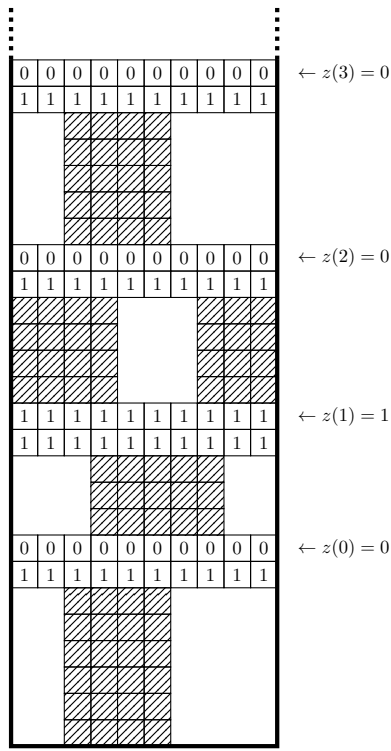

	This completes the construction and we can let $c_k:=\bigcup_n p^n_k$. Given any
	$A\in\mathcal{A}$, the columns in $A$ were given a chance to meet all of the dense subsets
	of $\Add(\omega,A)$, and so $\set{c_k}{k\in A}$ are mutually generic over $M$, and consequently they are amalgamable.
	
	On the other hand, suppose we are given $\set{c_k}{k\in B}$ for some
	$B\notin\mathcal{A}$. Observe that the only places where all of the columns in $B$ are active
	are exactly the coding points for $z$, since $\mathcal{A}$ is closed under subsets, so in between the coding
	points $B$ always has to have a column padded with $0$s. It follows
	that, given the reals $c_k$ for $k\in B$, we can recover the catastrophic real $z$,
	which means that the corresponding extensions are not amalgamable.
\end{proof}

It is not difficult to see that any finite poset embeds into a family of sets of natural numbers
closed under subsets in a way that preserves the nonexistence of upper bounds. It therefore follows
from Mostowski's theorem that any finite poset embeds into the generic multiverse while preserving
nonamalgamability. This has an interesting consequence for the theory of the generic multiverse.

\begin{corollary}
	\label{cor:ExistsTheoryDecidable}
	For each $0<n<\omega$ let $A_n(x_0,\dotsc,x_n)$ be the predicate in the language of
	a partial ordering asserting that the elements $x_0,\dotsc,x_n$ have a common upper bound.
	Then the $\exists$-theory of the generic multiverse in the language 
	$\{\leq\}\cup\set{A_n}{0<n<\omega}$ is decidable.
\end{corollary}

\begin{proof}
	Consider an existential sentence in the specified language. It asserts that there
	are some objects $x_0,\dotsc,x_m$ of which some quantifier-free formula $\psi(\vec{x})$ holds.
	The formula $\psi$, in turn, says that some of these objects are ordered in particular ways
	and that certain collections of them do or do not have upper bounds.
	One can clearly check in a computable way whether
	these requirements are consistent. If they are, we can realize them in a finite poset and	
	our observation above implies that this pattern may be found in the generic multiverse, and therefore that the existential sentence we started with is true in the multiverse.
\end{proof}

In the next theorem we shall extend Mostowski's \autoref{thm:Mostowski} to also include certain 
infinite posets and also to have the embedding respect nonzero lower bounds.

\begin{definition}
	Let $P$ and $Q$ be posets with least elements $0_P$ and $0_Q$, respectively. An injective map
	$f\colon P\to Q$ is a \emph{$*$-embedding} if:

	\begin{itemize}
		\item $x\leq y \iff f(x)\leq f(y)$ for any $x,y\in P$,
		\item any finite subset $X$ of $P$ has an upper, or nonzero lower, bound in $P$
		if and only if $f[X]$ has an upper, or nonzero lower, bound in $Q$, respectively.\qedhere
	\end{itemize}
\end{definition}
It follows from the definition that a $*$-embedding preserves the
least element. Furthermore, the forward direction of the equivalences in the second bullet point
easily follows from the fact that $f$ is increasing.

\begin{definition}
	A family of sets $\mathcal{A}$ is \emph{defined by a set of obstacles} $\mathcal{B}$ on a set $I$ if $\mathcal{A}$ consists of the subsets of $I$ that do not contain any element of $\mathcal{B}$ as a subset.
\end{definition}

Note that any such family is closed under subsets, since if a set avoids all the obstacles, then so does any subset. For a given family $\mathcal{A}$ defined by obstacles, we may assume without loss that $I=\bigcup\mathcal{A}$, simply by replacing $I$ with $\bigcup\mathcal{A}$. This amounts to assuming that $\mathcal{A}$ contains all singletons from $I$. We may also assume that the obstacle sets are subsets of $I$ and that they all have at least two elements. So we will generally make these additional assumptions without remark. We shall be principally interested in the case where all the obstacles are finite.

\begin{theorem}
\label{thm:GeneralBlockchain}
	Suppose that a family of sets $\mathcal{A}$ is defined in $M$ by finite obstacles on a set $I$. 
	Then there are $M$-generic Cohen reals $\set{c_i}{i\in I}$, with the following properties:
	
	\begin{enumerate}
		\item If $A\in\mathcal{A}$, then $\seq{c_i}{i\in A}$ is generic for $\Add(\omega,A)$
		over $M$.
		\item If $B\in M$, $B\subseteq I$ and $B\notin\mathcal{A}$, then the family $\set{M[c_i]}{i\in B}$
		does not amalgamate in the generic multiverse.
		\item If $A,A'\in \mathcal{A}$ then 
		$M[c_i\mid i\in A]\cap M[c_i\mid i\in A']=M[c_i\mid i\in A\cap A']$.
	\end{enumerate}
\end{theorem}

In the hypotheses of the theorem we of course intend that the set $I$ as well as the
set of finite obstacles and $\mathcal{A}$ itself are elements of the model $M$. The same
will be true when we discuss corollaries and other versions of this result, such as
\autoref{thm:WideAmalgamability} and \autoref{thm:WideClassForcing}.

\begin{definition}
	Let $I$ be a set and $\P_i$ for $i\in I$ posets. If $p,q\in\prod_{i\in I}\P_i$ and
	$J\subseteq I$, we say that $p$ is a \emph{$J$-extension} of $q$, and write $p\leq_J q$, 
	if $p\leq q$ and $p\rest{I\setminus J}=q\rest{I\setminus J}$.
\end{definition}

Thus, a $J$-extension strengthens the condition only on the coordinates in $J$, leaving the other coordinates untouched. We emphasize that the relation $\leq_J$ is not simply the pullback of the order on 
$\prod_{i\in J}\P_i$.

Let us point out a useful fact about product forcing that we will rely on several times in the rest 
of the paper.
Suppose that $p$ is a condition in $\P=\prod_{i \in I}\P_i$, we have some $J\subseteq I$,
and $\phi$ is a $\Delta_0$~sentence in the forcing language of $\P_J=\prod_{i\in J} \P_i$.
It is straightforward to check that $p\forces_{\P}\phi$ if and only if 
$p\rest J\forces_{\P_J}\phi$. It follows from this that any condition in $\P$ has a
$J$-extension which decides $\phi$.
In the coming proofs we shall use this fact without further comment.

\begin{lemma}
\label{lemma:ProductDecision}
	Let $X\subseteq Y$ be sets and let $\P_y$ for $y\in Y$ be posets. Let $\P_Y$ be the product
	of the $\P_y$ (using whatever support you like) and let $\P_X$ be the subposet corresponding to 
	the coordinates	in $X$. Let $\chi$ be a $\P_Y$-name such that 
	$\P_Y\forces \chi\subseteq V[\Gdot \rest{X}]$ but $\P_Y\not\forces\chi\in V[\Gdot\rest{X}]$.
	Fix a condition $q\in\P_Y$ such that $q\forces \chi\notin V[\Gdot\rest{X}]$. Then there is a
	$\P_X$-name $\rho$ such that no $X$-extension of $q$ decides $\rho\in\chi$.
\end{lemma}

\begin{proof}
	First consider the case $X=\varnothing$. So suppose that for every $x$ the condition $q$ decides
	$\check{x}\in \chi$. Let $S=\set{x}{q\forces\check{x}\in\chi}$. Clearly $S$ is a ground-model
	set and $q\forces\chi=\check{S}$, which contradicts our assumption.
	
	In the general case, observe that
	\[q\rest X\forces (q\rest{Y\setminus X}\forces \chi\notin V[\Gdot\rest X])\]
	Applying the previous case in $V^{\P_X}$, we can find a $\P_X$-name $\rho$ such that
	\[q\rest X\forces(\textup{$q\rest{Y\setminus X}$ does not decide $\rho\in \chi$})\]
	We claim that this $\rho$ is as desired. So fix some $r\leq_X q$.
	Then it follows from the above that
	\[r\rest X\forces (\textup{$q\rest{Y\setminus X}$ does not decide $\rho\in\chi$})\]
	but, since $r\rest{Y\setminus X}=q\rest{Y\setminus X}$, we in fact get
	\[r\rest X\forces (\textup{$r\rest{Y\setminus X}$ does not decide $\rho\in\chi$})\]
	which implies that $r$ does not decide $\rho\in\chi$, as desired.
\end{proof}

\begin{proof}[Proof of \autoref{thm:GeneralBlockchain}]
\newcommand{\Abar}{\overline{\mathcal{A}}}
	We proceed with a blockchain argument very much like in the proof of \autoref{thm:Mostowski}.
	Some extra care is required since we can no longer terminate blocks by putting $1$s across
	an entire row, as this would prevent any infinite collection of the $c_i$ from being
	mutually generic. Instead we will work with blocks of finite width, and it is precisely the
	fact that $\mathcal{A}$ is defined by finite obstacles that will allow this approach to 
	succeed.
	
	Let $\mathcal{B}$ be the set of finite obstacles used to define $\mathcal{A}$. Fix a catastrophic real $z$ for $M$. We will be filling in an $\omega$-by-$I$
	matrix with $0$s and $1$s, and its columns will grow into the desired Cohen reals $c_i$.
	We will ensure at each stage of our construction that what we have constructed so far 
	is a condition in $\Add(\omega,I)$, in effect using finite support in the product. 
	Moreover, we will make sure that this condition is \emph{uniform}, in the sense that all of its 
	nonempty columns will have the same height.	Our construction will
	again be guided by some enumeration of all the relevant requirements in order type $\omega$. We will, moreover,
	assume that each requirement is listed infinitely often.
	At each stage the enumeration will give us either
	\begin{enumerate}
		\item an $A\in\mathcal{A}$ and an open dense subset of $\Add(\omega,A)$,
		\item a $B\in\mathcal{B}$, or
		\item a pair of sets $A,A' \in \mathcal{A}$,
		together with an $\Add(\omega,A)$-name $\sigma\in M$ and
		an $\Add(\omega,A')$-name $\tau\in M$ for subsets of $M[\Gdot\rest{A\cap A'}]$.
	\end{enumerate}
	Now suppose that we are at some stage of our construction, $p$ is the entire condition in
	$\Add(\omega,I)$ we have constructed so far, and we are given an open dense subset
	$D\subseteq\Add(\omega,A)$ for some $A\in\mathcal{A}$. In this case we simply
	extend the columns of $p$ in $A$ to meet the dense set $D$ and then pad with $0$s, 
	if necessary, to obtain a uniform condition.
	
	If we are given an obstacle $B\in\mathcal{B}$, we extend the columns of $p$ in $B$ by appending a row of $1$s and then the 
	bit $z(n)$, where $n$ is the number of times
	that $B$ has occurred in the enumeration before this step.
	Note that the result is still a condition in $\Add(\omega,I)$ since $B$ is finite.
	At the end we pad with $0$s to make the condition uniform, if necessary.
	
	Finally, suppose we are given a pair of names $\sigma$ and $\tau$ as above. If
	$p\rest A\forces \sigma\in M[\Gdot\rest{A\cap A'}]$ we do nothing in this step and proceed with
	the rest of the construction. If this is not the case, we will attempt to find an extension $q$ of
	$p$ that forces $\sigma\neq\tau$.
	Given our assumption, we can find some $p'\leq_A p$ such that
	$p'\rest{A}\forces \sigma\notin M[\Gdot\rest{A\cap A'}]$. Now apply \autoref{lemma:ProductDecision}
	with $X=A\cap A', Y=A, q=p'\rest A,$ and $\chi=\sigma$ to get a name $\rho$ such that no
	$(A\cap A')$-extension of $p'\rest{A}$ decides $\rho\in\sigma$. 
	Pick an $A$-extension of $p'$ that decides $\rho\in\sigma$ and pad it with $0$s
	to make it uniform; let $p''$ be the resulting condition. Now consider two cases:
	\begin{enumerate}
		\item If $p''\rest{A'}$ does not decide $\rho\in\tau$ the same way as $p''\rest A$
		decides $\rho\in\sigma$, we can find a $q\leq_{A'} p''$ such that $q\rest{A'}$ decides
		it the opposite way. In this case we pad $q$ with $0$s to make it uniform, and the resulting
		condition is the next step in our construction.
		\item If $p''\rest{A'}$ decides $\rho\in\tau$ the same way as $p''\rest{A}$ decides
		$\rho\in\sigma$, observe that $p'\rest{A} \cup p''\rest{A\cap A'}$ is an
		$(A\cap A')$-extension of $p'\rest{A}$, so, by our assumption, it does not decide
		$\rho\in\sigma$. Therefore there must be some $q\leq_A p'\cup(p''\rest{A'})$
		so that $q\rest{A}$ decides $\rho\in\sigma$ differently than
		$q\rest{A'}\leq p''\rest{A'}$ decides $\rho\in\tau$. We take this $q$ to be the next
		step in our construction, after possibly padding with $0$s to make it uniform.
	\end{enumerate}
	Note that this $q$ is as desired: Since $q$ decides $\rho\in\sigma$ and $\rho\in\tau$
	in opposite ways, it definitely forces $\sigma\neq\tau$.
	
	This finishes the construction of the reals $c_i$. Let us verify that they satisfy all
	of the properties in the theorem. If $A\in\mathcal{A}$ then, in the course of our
	construction, we considered every open dense subset of $\Add(\omega,A)$ in $M$, and we made sure
	that the columns of our matrix in $A$ met these dense sets. Therefore
	$\seq{c_i}{i \in A}$ really is generic for $\Add(\omega,A)$.
	
	If $B\notin \mathcal{A}$, then it must contain an obstacle $B'\in\mathcal{B}$, and we handled this $B'$
	infinitely often during our construction, each time adding a row of $1$s across the columns
	in $B'$ and the next bit of $z$. Furthermore, we made sure in our construction to only
	add a row of $1$s across a set of columns if the set was indeed an obstacle set (in which case we also
	coded a bit of $z$) or if it was covered by an element of $\mathcal{A}$. 
	Since the latter is not the case for $B'$, we can recover $z$ from the reals
	$\set{c_i}{i \in B'}$, and so the family $\set{M[c_i]}{i \in B'}$ is not amalgamable
	and therefore neither is $\set{M[c_i]}{i \in B}$.
	
	Finally, suppose that $A,A'\in\mathcal{A}$ and we want to see that
	$$M[c_i\mid i \in A]\cap M[c_i\mid i \in A']\subseteq M[c_i\mid i \in A\cap A']$$
	Let $x$ be an element of this intersection; we may assume by $\in$-induction that
	$x\subseteq M[c_i\mid i \in A\cap A']$. Pick an $\Add(\omega,A)$-name $\sigma$ and an 
	$\Add(\omega,A')$-name $\tau$ for $x$. But we were handed this pair of names at some
	stage of our construction. At that point it must have been the case that the condition we had
	forced $\sigma\in M[\Gdot\rest{A\cap A'}]$, since otherwise we would have made sure that
	$\sigma$ and $\tau$ would be interpreted as different sets. But this means that
	$x\in M[c_i\mid i\in A\cap A']$, as required.
\end{proof}

\begin{theorem}
\label{thm:EmbedFamilyOfSetsIntoCohenMultiverse}
	Let $I$ and $\mathcal{A}$ be as in \autoref{thm:GeneralBlockchain}. Then $(\mathcal{A},\subseteq)$
	$*$-embeds into the generic multiverse given by posets of the form 
	$\Add(\omega,X)$ for a set $X$ (and mapping $\varnothing$ to $M$).
\end{theorem}

\begin{proof}
\newcommand{\Mext}[1]{M[c_i\mid i\in #1]}
	Let $\set{c_i}{i\in I}$ be the Cohen reals given by \autoref{thm:GeneralBlockchain}.
	We define a map on $\mathcal{A}$ by sending $A\in\mathcal{A}$ to
	$\Mext{A}$. This map is clearly increasing. Let us check that it is also a $*$-embedding.
	
	Suppose that $\Mext{A}\subseteq \Mext{A'}$ for some $A,A'\in\mathcal{A}$. It follows
	that the family $\set{M[c_i]}{i\in A\cup A'}$ is amalgamable, so $A\cup A'\in\mathcal{A}$.
	In particular, $\seq{c_i}{i\in A\cup A'}$ are mutually generic over $M$.
	Now suppose that there is some $\ell\in A\setminus A'$. Then $c_\ell$ would have to be
	generic over $\Mext{A'}\supseteq\Mext{A}$, which is impossible. Therefore $A\subseteq A'$.
	
	Now suppose that $X$ is a finite subset of $\mathcal{A}$ and that
	$\set{\Mext{A}}{A\in X}$ is amalgamable. It follows as above that $\bigcup X\in\mathcal{A}$,
	so $X$ does have an upper bound in $\mathcal{A}$.
	
	Finally, with $X$ as above, suppose that $\set{\Mext{A}}{A\in X}$ has a lower bound
	in the generic multiverse strictly above $M$. But recall that one of the properties of
	the Cohen reals constructed in \autoref{thm:GeneralBlockchain} is that
	$\bigcap_{A\in X}\Mext{A}=\Mext{\bigcap X}$. It follows that $\bigcap X$ is nonempty,
	and so it is a nonzero lower bound for $X$ in $\mathcal{A}$.
\end{proof}

The theorem is actually slightly stronger than stated, since we obtain nonamalgamability
in the entire generic multiverse and not just in the restricted multiverse given by adding Cohen 
reals. On a related note, we can tweak the real $z$ that we coded between the generics to
achieve different kinds of nonamalgamability. For example, we could take $z$ to be a sufficiently 
generic random real over $M$, in which case the models given will be nonamalgamable in the
Cohen multiverse, but amalgamable in, say, the \textsc{ccc}-multiverse.

\begin{corollary}
\label{cor:FinitePosetsEmbed}
	Any finite poset with a least element $*$-embeds into the generic multiverse. Moreover,
	any finite meet-semilattice $*$-embeds as such into the generic multiverse.
\end{corollary}

\begin{proof}
	It is easy to check that the composition of $*$-embeddings is itself a $*$-embedding, so,
	in view of \autoref{thm:EmbedFamilyOfSetsIntoCohenMultiverse}, we only need to check
	that any finite poset $P$ $*$-embeds into a family of sets $\mathcal{A}$ as in
	\autoref{thm:GeneralBlockchain}. We can take $\mathcal{A}$ to consist of all sets of the 
	form $p{\downarrow}:=\set{q\in P}{0_P<q\leq p}$ and their subsets, where $p\in P$ and $0_P$ is
	the least element of $P$, and it is straightforward to see that the map $p\mapsto p{\downarrow}$
	is a $*$-embedding.
	If $P$ is in addition a meet-semilattice, we can furthermore observe that both the map
	$p\mapsto p{\downarrow}$ and the $*$-embedding from 
	\autoref{thm:EmbedFamilyOfSetsIntoCohenMultiverse} preserve meets.
\end{proof}

Arguing exactly as in \autoref{cor:ExistsTheoryDecidable}, we can now deduce the decidability
of the existential theory of the generic multiverse augmented with the amalgamability predicates
$A_n$ and a partial meet operation (seen as a ternary predicate).

\begin{corollary}
	\label{cor:ExistsSemilatticeTheoryDecidable}
	Let the predicates $A_n$ be as in \autoref{cor:ExistsTheoryDecidable}.
	Then the $\exists$-theory of the generic multiverse in the language
	$\{\leq,\wedge\}\cup\set{A_n}{0<n<\omega}$ is decidable.
\end{corollary}

\section{Generalizing to Wide Forcing}
\label{sec:WideForcing}

The blockchain arguments given in the previous section clearly generalize to other forcing notions 
where conditions are some kind of bounded sequence, such as $\Add(\kappa,1)$ or 
$\Coll(\kappa,\lambda)$.
The next question we are interested in is whether the nonamalgamability phenomenon (and various 
$*$-embeddings) can be realized using an even larger variety of forcing notions, and even possibly 
while using several nonisomorphic posets at once.

The obvious obstacle to our grand ambition is that there are always pairs of posets such that
any two generic extensions of $M$ by these posets amalgamate: just take $\Add(\omega,1)$ and
$\Add(\omega_1,1)$, for example. Retreating a bit to just consider one poset at a time, some
forcing notions $\P$ exhibit \emph{automatic mutual genericity}, meaning that any two distinct
generic filters for $\P$ are mutually generic. Any two extensions by such forcing are therefore amalgamable. An early example of this phenomenon
was Jensen's forcing to add a $\Delta^1_3$ real of minimal degree over $L$ 
(see~\cite{Jensen1970:DefinableSetsOfMinimalDegree}). Slightly more generally, the second author
showed in~\cite{Hamkins2016:UpwardClosureAndAmalgamationInGenericMultiverse} that, under
$\diamondsuit$, there is a Suslin tree exhibiting automatic mutual genericity.
The question whether posets exhibiting automatic mutual genericity always
exist is open
(see \cite{Hamkins2015.MO222602:Is-it-consistent-with-ZFC-that-no-nontrivial-forcing-has-automatic-mutual-genericity?}).

Skirting these counterexamples, the second author introduced the notion of a wide poset and proved the nonamalgamation theorem for wide forcing.

\begin{definition}[\cite{Hamkins2016:UpwardClosureAndAmalgamationInGenericMultiverse}]
	A poset $\P$ is \emph{wide} if there is, for any $p\in\P$, a maximal antichain
	in $\P$ below $p$ of size $|\P|$.
\end{definition}

For example, the Cohen poset $\Add(\omega,1)$ is wide, as are collapse posets 
$\Coll(\kappa,\lambda)$, Sacks forcing, Mathias forcing, etc. It is readily seen that
the class of wide posets is closed under finite products, finite lottery sums, provided
that the factors are all of the same size, and finite iterations (although some care is needed in
the specific presentation of the iteration).

The second author went on to show that, given
posets $\P_0$ and $\P_1$ that are wide and of equal size in $M$, there are generic filters 
$G_0$ and $G_1$
such that the extensions $M[G_0]$ and $M[G_1]$ do not amalgamate. Let us now generalize this argument to prove a version of \autoref{thm:GeneralBlockchain} in the setting of wide posets.

\begin{theorem}
\label{thm:WideAmalgamability}
	Suppose that a family of sets $\mathcal{A}$ is defined in $M$ by finite obstacles on a set $I$ and $\set{\P_i}{i\in I}\in M$ is a family of wide posets in $M$, all of the same size $\kappa\geq|I|$. 	
	Then there are generic filters $G_i\subseteq\P_i$ over $M$ with the following properties:
	\begin{enumerate}
		\item \label{i:generic}
		If $A\in\mathcal{A}$ then $\prod_{i\in A}G_i$ is generic for $\prod_{i\in A}\P_i$
		over $M$.
		\item \label{i:nonamalg}
		If $B\in M$, $B\subseteq I$ and $B\notin\mathcal{A}$, then the family $\set{M[G_i]}{i\in B}$ does
		not amalgamate in the generic multiverse.
		\item \label{i:intersect}
		If $A^0,A^1\in \mathcal{A}$ then $M[\prod_{i\in A^0}G_i]\cap M[\prod_{i\in A^1}G_i]=
		M[\prod_{i\in A^0\cap A^1}G_i]$.
	\end{enumerate}
	The products above are to be seen as subposets of the product $\prod_{i\in I}\P_i$ which may use 
	supports from an arbitrary ideal in $M$ extending the finite ideal on $I$.
\end{theorem}

\begin{proof}
	For $X$ a subset of $I$ we write $\P_X=\prod_{i\in X}\P_i$ and similarly for $G_X$.
	For each $i\in I$ we fix an
	enumeration of $\P_i$ in order type $\kappa$ in $M$. Since the $\P_i$ are wide in $M$,
	we can also find for each $p\in \P_i$ a maximal antichain $Z(p)\subseteq \P_i$ below $p$
	of size $\kappa$, together with an enumeration in order type $\kappa$.
	Finally, fix a catastrophic real $z$ for $M$.
	
	As in the blockchain construction of \autoref{sec:Blockchain} we will undertake a
	construction of a descending sequence of conditions
	$p_n=\seq{p_n^i}{i\in I}\in \prod_{i\in I}\P_i$ whose restriction to each coordinate $i$
	will generate the generic filter $G_i$. We will be guided by some enumeration of all the
	relevant parameters in order type $\omega$. We arrange matters such that at each stage of the 
	construction we are given the following: two sets $A^0,A^1\in\mathcal{A}$, a dense subset
	$D\subseteq \P_{A^0}$ in $M$, and a $\P_{A^0}$-name $\sigma\in M$ and a $\P_{A^1}$-name 
	$\tau\in M$, both for subsets of $M[\Gdot_{A^0\cap A^1}]$ (of course, we want to run through
	all the possible combinations of these parameters).
	
	Before we start the construction, let us briefly describe the idea. The key property of the wide
	posets $\P_i$ will be that, given a condition $p\in \P_i$, a generic $G_i$ containing $p$ 
	picks out a
	unique element of the antichain $Z(p)$, which amounts to, in view of our fixed enumeration, a 
	particular ordinal below $\kappa$. We could, for example, use this mechanism to encode the
	bits of the catastrophic real $z$ as we were building our sequence of conditions $p_n$.
	This approach would be similar to what we did in the proof of \autoref{thm:GeneralBlockchain}.
	However, since the structure of the posets $\P_i$ may be more complicated than that of Cohen
	forcing, we might, along the way, lose track of the \emph{coding points}---those conditions
	whose corresponding antichains we should consult to decode $z$. To remedy this, we will
	keep track of the coding points used at each stage and code not only the catastrophic real $z$
	but also the next set of coding points. In addition, at
	various points in the construction we will want to record one of three additional special symbols
	(say $\sharp$, $\flat$, and $\natural$) which will serve to mark which case of the construction
	we are in. These special symbols will be used in the decoding process to signal which generic
	to look at to recover further information.
	
	We start the construction by letting $p_0^i$ be the top condition of $\P_i$ for each 
	$i\in I$. Now suppose that we have constructed $p_n$ and are given 
	$A_n^0,A_n^1\in\mathcal{A}$, a dense $D\subseteq \P_{A_n^0}$, and names $\sigma$ and $\tau$.
	Let $B_n^0$ and $B_n^1$ be the sets of the first $n$ many elements of $I\setminus A_n^0$
	and $I\setminus A_n^1$, respectively, in some fixed enumeration of $I$ in $V$ in order type
	$\omega$. We also assume that we have, along the way, built a sequence of conditions
	$c_m\in \P_I$ which are the coding points used at stage $m$. To be precise,
	for each $m<n$ and each $i\in B^0_m$ we assume that $p_{m+1}^i\leq c_m^i$ is the 
	condition in the antichain below $c_m^i$ whose index codes: the sets $B^0_{m+1}$ and
	$B^1_{m+1}$, the bit $z(m)$, and the sequence $\seq{c^j_m}{j\in B^0_{m+1}}$
	(note that 	all of these are essentially finite	subsets of $\kappa$, so it makes sense to code
	them together by a single ordinal below $\kappa$). 
	If we additionally make sure that $p^i_{m+1}=c^i_m$ for $i\notin B^0_m$, then by knowing
	$c_m\rest{B^0_m}$ and a generic $G_i$ for some $i\in B^0_m$, we can recover
	$p_{m+1}\rest{B^0_{m+1}}$.
	
	Returning to our construction at stage $n$, it remains for us to build $c_n\leq p_n$
	in such a way that $c_n\rest{B^0_n}$ will be recoverable from $p_n\rest{B^0_n}$ and
	a well-chosen generic.
	First of all, we can find $q_n\leq_{A^0_n} p_n$ such that $q_n\rest{A^0_n}\in D$
	and $q_n$ decides whether $\sigma\in M[\Gdot_{A^0_n\cap A^1_n}]$. We now have a small tree of
	cases to consider.

	\textbf{Case 1:} Suppose that $q_n\forces \sigma\in M[\Gdot_{A^0_n\cap A^1_n}]$. 
	In this case we just let $c_n\leq q_n$ be the 
	extension obtained by coding the symbol $\sharp$ below the conditions on the coordinates
	$i\in B^0_n$ and keeping the other coordinates.
	In particular, since $B^0_n$ is disjoint from $A^0_n$, once we know $p_n\rest{B^0_n}$
	and a generic $G_i$ for some $i\in B^0_n$, we can recover $c_n\rest{B^0_n}$.
	
	\textbf{Case 2:} Suppose that $q_n\forces \sigma\notin M[\Gdot_{A^0_n\cap A^1_n}]$. In this
	case we will attempt to find a condition $c_n$ which forces $\sigma\neq\tau$. 
	We can apply \autoref{lemma:ProductDecision} and get a name $\rho$ for an element of
	$M[\Gdot_{A^0_n\cap A^1_n}]$ such that no 
	$(A^0_n\cap A^1_n)$-extension of $q_n$ decides whether $\rho\in\sigma$. It will be this $\rho$
	that we will try to use to distinguish $\sigma$ and $\tau$ in the extension. 
	We find ourselves in one of two subcases.
	
	\textbf{Case 2a:} Suppose that $q_n$ decides whether $\rho\in\tau$. 
	We then let $q_n'\leq_{A^0_n} q_n$
	be any extension which decides $\rho\in\sigma$ in the opposite way. After this we again
	let $c_n\leq q_n'$ be the extension obtained by coding the symbol $\sharp$
	below the conditions on the coordinates $i\in B^0_n$ and keeping the other coordinates.
	Just as before, knowing $p_n\rest{B^0_n}$ and some $G_i$ for $i\in B^0_n$ allows us
	to recover $c_n\rest{B^0_n}$.
	
	\textbf{Case 2b:} Suppose that $q_n$ does not decide whether $\rho\in\tau$.
	We first find $p'_n\leq_{A^0_n}q_n$ which forces $\rho\in\sigma$ and then define $q'_n\leq p'_n$
	by coding the 	sequence $p'_n\rest{B^1_n}$, together with the symbol $\flat$, below each
	condition $p'^i_n$ for $i\in B^0_n$ and fixing the other coordinates. Observe that, at
	this point, we can recover $q'_n\rest{B^0_n\cup B^1_n}$ from $p_n\rest{B^0_n}$ and some generic 
	$G_i$ for $i\in B^0_n$. But we are not yet finished, and we must ask whether we can commit to
	forcing $\rho \not \in \tau$ or whether we have to roll back some mistakes.
	
	\textbf{Case 2b (i):} Suppose that $q'_n\not\forces \rho\in\tau$. In this case we find $p''_n\leq_{A^1_n} q'_n$ which forces that $\rho\notin\tau$ and
	then let $c_n$ be the condition obtained from $p''_n$ by coding the sequence
	$p''_n\rest{B^0_n}$ below the conditions on the coordinates $i\in B^1_n$ and fixing the other
	coordinates. Note that $c_n\rest{B^0_n}$ can be recovered from $q'_n\rest{B^1_n}$
	and some generic $G_i$ for $i\in B^1_n$. 
	
	\textbf{Case 2b (ii):} Suppose that $q'_n\forces \rho\in\tau$. Recall that we also
	have $q'_n\forces \rho\in\sigma$, so this does not seem to be a good way to go if we
	want to end up forcing $\sigma\neq\tau$. So we are going to try to undo past mistakes
	and work instead with the condition 
	$q''_n=q_n\rest{A^0_n\setminus A^1_n}\cup q'_n \rest{I\setminus(A^0_n\setminus A^1_n)}$.
	The key property of this condition is that it still forces $\rho\in\tau$
	but it does not decide $\rho\in\sigma$ since it is essentially an $(A^0_n\cap A^1_n)$-extension
	of $q_n$. Therefore we can find $p'''_n\leq_{A^0_n}q''_n$ which
	forces $\rho\notin\sigma$ and then let $c_n\leq p'''_n$ be the condition obtained by
	coding the sequence $p'''_n\rest{B^0_n}$, together with the symbol $\natural$ (cancelling out
	the $\flat$ from before), below the conditions on the coordinates
	$i\in B^0_n$ and keeping the other coordinates. Then we can recover $c_n\rest{B^0_n}$
	from $q'_n\rest{B^0_n}$ and some generic $G_i$ for $i\in B^0_n$.
	
	This finishes the construction of $c_n$ and $p_{n+1}$. 
	It still remains for us to show that the generic filters $G_i$ obtained this
	way have the desired properties. Of these, property~(\ref{i:generic}) is easiest, since it is
	clear that we have ensured genericity by running through all the dense sets corresponding
	to any particular $A\in\mathcal{A}$. Property~(\ref{i:intersect}) is also not difficult to see.
	Suppose that we are given an element $x\in M[G_{A^0}]\cap M[G_{A^1}]$ for some 
	$A^0,A^1\in\mathcal{A}$ and we wish to show that $x\in M[G_{A^0\cap A^1}]$. 
	Let $\sigma$ and $\tau$ be the $\P_{A^0}$- and $\P_{A^1}$-names for $x$, respectively. 
	We may assume, by $\in$-induction, that $x$ is in fact a subset of $M[G_{A^0\cap A^1}]$.
	These two names were considered at some stage of the construction, at which time one of two things
	occurred: Either it was forced at that stage that $\sigma\in M[G_{A^0\cap A^1}]$, or we arranged 
	matters so that	it was forced that $\sigma\neq\tau$. But the second option could not have happened,
	since the two names both evaluate to $x$. It therefore follows that $x\in M[G_{A^0\cap A^1}]$.
	
	Finally we turn to property~(\ref{i:nonamalg}). 
	Suppose that we have a set $B\in M$ which does not belong to $\mathcal{A}$ and we
	wish to show that the family $\set{M[G_i]}{i\in B}$ does not amalgamate. Since $\mathcal{A}$ is defined by finite obstacles, it suffices to show this for the finite obstacle sets $B$. We will argue that,
	given the filter $G_B$ and some parameters in $M$, we can recover the catastrophic real $z$,
	which prevents amalgamability. Since $B$ is finite, there is some $n$ so that for all
	larger $m$ we have $B\subseteq A^0_m\cup B^0_m$ and $B\subseteq A^1_m\cup B^1_m$.
	Note that also, since $\mathcal{A}$ is closed under subsets, the intersections $B\cap B^0_m$
	and $B\cap B^1_m$ are nonempty. We may assume that we already know the part of $z$ up to $n$,
	as well as $B^0_n, B^1_n$, and $p_n\rest{B^0_n}$, and we are looking to recover
	$z(n),B^0_{n+1},B^1_{n+1}$, and $p_{n+1}\rest{B^0_{n+1}}$.
	
	Fix indices $i^0\in B\cap B^0_n$ and $i^1\in B\cap B^1_n$. First we check to see what symbol
	is coded at the place where the filter $G_{i^0}$ meets the antichain $Z({p_n^{i^0}})$. 
	If we see a $\sharp$, we know that we were either in Case 1 or 2a at this stage of the 
	construction and we can immediately recover $c_n\rest{B^0_n}$, and this condition,
	together with $G_{i^0}$ again, codes all the information we required. 
	Similarly, if we see a $\flat$, we know that the construction went through Case 2b.
	We should compute $q_n'\rest{B^0_n\cup B^1_n}$, as described above, and see whether the 
	$\flat$ is followed by a $\natural$: if so, we know that
	we proceeded through Case 2b (ii) and we use $G_{i^0}$ to recover $c_n\rest{B^0_n}$,
	and if not, we proceeded through Case 2b (i) and we use $G_{i^1}$ in the same way.
\end{proof}

In the exact same way as we proved \autoref{thm:EmbedFamilyOfSetsIntoCohenMultiverse} using
\autoref{thm:GeneralBlockchain}, we can use \autoref{thm:WideAmalgamability} to prove the following. 

\begin{theorem}
\label{cor:EmbeddingIntoWideMultiverse}
	Let $I,\mathcal{A}$, and the posets $\P_i$ be as in \autoref{thm:WideAmalgamability}.
	Then $(\mathcal{A},\subseteq)$ $*$-embeds into the generic multiverse given by products of the posets
	$\P_i$ (and mapping $\varnothing$ to $M$).
\end{theorem}

\section{(Non-)Amalgamability in the Class-Generic Multiverse}
\label{sec:ClassWideForcing}

In this section we turn our attention to the context of different possible collections of classes for a fixed countable model of \ZFC, the original context in which Mostowski \cite{Mostowski1976:ModelsOfBGAxioms} was interested. Specifically, he was interested in the structure
of models of G\"odel--Bernays set theory whose first-order parts are all the same countable transitive model of \ZFC. He investigated the possible patterns of (non-)amalgamability of these models. Just as we generalized his argument in the set forcing case, we can do the same for class forcing. The results of this section are an adaptation of the results about wide forcings in \autoref{sec:WideForcing}.

We will treat models of second-order set theory as two-sorted structures $(M,\mathcal{X},\in^{(M,\mathcal{X})})$ with sets or first-order part $M$ and classes or second-order part $\mathcal{X}$. We will suppress writing the membership relation, referring simply to $(M,\mathcal{X})$. We may always assume that $\mathcal{X}$ is a collection of subsets of $M$. If $M$ is transitive with the true $\in$ as its membership relation, then so is $\mathcal{X}$ and the set--set and set--class membership relations of $(M,\mathcal{X})$ are the true $\in$. 

\begin{definition}
G\"odel--Bernays set theory \GBC with the axiom of Global Choice is the second-order set theory axiomatized by \ZFC for sets, Class Extensionality, Class Replacement, Global Choice (in the form there is a class bijection $\Ord \to V$), and Elementary Comprehension (i.\,e.\ Comprehension for formulae with only set quantifiers (but allowing class parameters)). Dropping Powerset from the axioms for sets gives the theory \GBCm.\footnote{In the absence of Powerset we have that Collection is stronger than Replacement \cite{zarach1996}. We want \GBCm to be axiomatized with Collection, as the version with only Replacement is badly behaved \cite{gitman-hamkins-johnstone2016}. Also, the various forms of Global Choice require Powerset to prove their equivalence. Without Powerset, the existence of a class bijection $\Ord \to V$ is the strongest. We will use this strongest form of Global Choice in the proof of \autoref{thm:WideClassForcing}.}
\end{definition}

To make the argument work we will need a slight strengthening of \GBC in the ground model.

\begin{definition}
Elementary Class Choice \ECC is the schema asserting that if for every set there is a class witnessing some first-order property, then there is a single class coding witnesses for each set. Formally, let $\phi(x,Y,A)$ be a first-order formula, possibly with a set or class parameter $A$. The instance of \ECC for $\phi(x,Y,A)$ asserts that if for every set $x$ there is a class $Y$ such that $\phi(x,Y,A)$ then there is a class $C$ such that for every set $x$ we have that $\phi(x,(C)_x,A)$ holds, where $(C)_x = \{ y : (x,y) \in C \}$ is the $x$-th slice of $C$.
\end{definition}

\begin{theorem}[Gitman--Hamkins \cite{gitman-hamkins2015}\protect\footnotemark]
\footnotetext{This paper has not appeared in publication, but slides can be found online \cite{Gitman2014.webpost:Kelley-Morse-set-theory-and-choice-principles-for-classes}.}
Kelley--Morse set theory\footnote{Recall that Kelley--Morse set theory is axiomatized by the axioms of \GBC plus the full Comprehension schema for all formulae, even those with class quantifiers.} does not prove \ECC. 
\end{theorem}

\begin{theorem}[{Williams \cite[Corollary 2.48]{williams-diss}}]
If $(M,\mathcal{X})$ is a model of \GBCm $+$ \ETR then there is $\mathcal{Y} \subseteq \mathcal{X}$ second-order definable such that $(M,\mathcal{Y})$ is a model of \GBCm $+$ \ETR $+$ \ECC.\footnote{See \cite{gitman-hamkins2016} for a definition of Elementary Transfinite Recursion \ETR, originally due to Fujimoto \cite{fujimoto2012}. The theory \GBC $+$ \ETR is stronger than \GBC in consistency strength, but weaker than Kelley--Morse set theory and indeed weaker than \GBC $+$ $\Pi^1_1$-Comprehension.}
As a consequence, the following pairs of theories are equiconsistent:
\begin{itemize}
\item \GBC $+$ \ETR and \GBC $+$ \ETR $+$ \ECC.
\item \GBCm $+$ \ETR and \GBCm $+$ \ETR $+$ \ECC.
\end{itemize}
\end{theorem}

Next we need a definition of wide forcing specialized to the context of class forcing.

\begin{definition}
A class forcing notion $\P$ is {\em $\Ord$-cc} (synonymously, {\em has the $\Ord$-chain condition}) if every antichain of $\P$ is set-sized.
\end{definition}

\begin{definition}
A class forcing notion $\P$ is {\em $\Ord$-wide} if for every condition $p \in \P$ we have that $\P$ restricted below $p$ is not $\Ord$-cc. That is, below any condition in $\P$ there is a proper-class-sized antichain.
\end{definition}

For the generic multiverse of a model of second-order set theory we want to confine the collection of forcing notions we allow to only the pretame forcings.\footnote{See \cite{friedman:book} for a definition of pretameness.}
The reason is two-fold. First, pretameness is equivalent to the preservation of \GBCm \cite{stanley1984,friedman:book}. Indeed, if a class forcing is not pretame then any forcing extension by it will fail to satisfy Replacement \cite{HKS2018}. This is problematic for nonamalgamability arguments, as it is Replacement which allows us to argue that a forcing extension of $M$ cannot contain a catastrophic real for $M$. Second, \GBC is not strong enough to prove that every nonpretame forcing notion admits a forcing relation \cite{HKLNS2016}. With this in mind, say that the \emph{class-generic multiverse} of a countable transitive model of second-order set theory $(M,\mathcal{X})$ is the smallest collection of models containing $(M,\mathcal{X})$ and closed under grounds and extensions by pretame class forcings. Similar to the set-forcing case, a family of universes in the generic multiverse amalgamates if there is some universe in the generic multiverse which contains all of them.

Given a class forcing notion $\P \in \mathcal{X}$ and a filter $G \subseteq \P$ generic over $(M,\mathcal X)$, meaning that $G$ meets every dense subclass of $\P$ in $\mathcal{X}$, we will write $(M,\mathcal{X})[G]$ for the generic extension of $(M,\mathcal X)$ by $G$. This is defined in the usual way, with the sets of the extension being the interpretations of set $\P$-names in $M$ and the classes of the extension being the interpretations of class $\P$-names in $\mathcal{X}$. If $\P$ is pretame then it follows from the preservation of $\GBCm$ that the sets in the extension are precisely the classes which are subclasses of sets. 

We wish to emphasize that the geology for the class-generic multiverse is not as nice as that for the set-generic multiverse. For example, to argue in \autoref{lemma:SubsetEqualsForcingExtension} that inclusion in the (set) generic multiverse is the same as being a forcing extension, we appealed to the intermediate model theorem. But it is a folklore result that the intermediate model theorem can fail for (tame) class forcings. Nevertheless, these difficulties will not concern us here. We are looking only at forcing extensions of a fixed model, not forcing extensions of grounds of forcing extensions of ... And our method of ensuring nonamalgamability---coding a catastrophic real---rules out amalgamability in the broader generic multiverse, not just among forcing extensions of a fixed model.

The primary concern is with models of \GBC, but our argument nowhere uses the axiom of powerset, so we state the theorem in a more general setting.

\begin{theorem}
\label{thm:WideClassForcing}
Assume $(M,\mathcal{X})$ is a countable transitive model of \GBCm $+$ \ECC, meaning that both $M$ and $\mathcal{X}$ are countable, and suppose a family of sets $\mathcal{A}$ is defined in $(M,\mathcal{X})$ by finite obstacles on a class $I\in\mathcal{X}$. 
Let $\set{\P_i}{i \in I} \in \mathcal{X}$ be a family of $\Ord$-wide forcing notions in $(M,\mathcal{X})$.\footnote{To clarify: the $\P_i$ are proper classes so the literal collection of all of them is too high in rank to be a class. But we can code this collection as a single class via pairs, say using the class $C = \{ (i,p) : i \in I \text{ and } p \in \P_i \}$ so that $\P_i = \{ p : (i,p) \in C \}$ is the $i$-th slice of $C$. And we must do something similar for sequences of classes, as we will use below.}
Then there are $G_i \subseteq \P_i$ generic over $(M,\mathcal{X})$ satisfying the following three properties.
\begin{enumerate}
\item If $A \in \mathcal{A}$, then $\prod_{i \in A} G_i$ is generic for $\prod_{i \in A} \P_A$ over $(M,\mathcal{X})$.
\item If $B \subseteq I$ is in $\mathcal{X}$ with $B \not \in \mathcal{A}$, then the family $\{ (M,\mathcal{X})[G_i] \mid i \in B \}$ does not amalgamate in the generic multiverse.
\item If $A^0,A^1 \in \mathcal A$ then $M[\prod_{i \in A^0} G_i] \cap M[\prod_{i \in A^1} G_i] = M[\prod_{i \in A^0 \cap A^1} G_i]$. 
\end{enumerate}
The products above are to be seen as suborders of the product $\prod_{i \in I} \P_i$, which may use supports from any ideal in $(M,\mathcal{X})$ which extends the finite ideal.
\end{theorem}

\begin{proof}[Proof sketch]
This is proved much the same as \autoref{thm:WideAmalgamability}. We give a sketch of the proof here, making clear where \ECC and the strong form of Global Choice are used. 

Fix the following notation: For classes $A \subseteq I$ let $\P_A = \prod_{i \in A} \P_i$ and similarly for $G_A$. By Global Choice fix an enumeration of $M$ in order type $\Ord^M$. Using \ECC, pick for each $p \in \P$ a proper-class-sized ($=$ $\Ord^M$-sized) maximal antichain $Z(p) \subseteq \P_i$ below $p$. Externally to $(M,\mathcal{X})$, fix $z$ a catastrophic real for $M$, an enumeration of $I$ in order type $\omega$, and an enumeration of all the relevant parameters from $M$ in order type $\omega$. We will build a descending $\omega$-sequence of conditions $p_n = \seq{p^i_n}{i \in I} \in \P_I$ whose restrictions to the $i$-th coordinate will generate $G_i$. At each stage in the construction we deal with $A^0, A^1 \in \mathcal{A}$, a dense subclass $D$ of $\P_{A^0}$ in $\mathcal{X}$, a $\P_{A^0}$-name $\sigma \in \mathcal X$, and a $\P_{A^1}$-name $\tau \in \mathcal X$, both for subsets or subclasses of $(M,\mathcal X)[G_{A^0 \cap A^1}]$. 

As before, each $G_i$ is determined by a descending sequence of conditions $\seq{p_i}{i \in \omega}$. We ensure that $G_A$ is generic for $A \in \mathcal A$ by meeting every dense subclass. And the key point of our construction will be to find \emph{coding points} which can then be used to guide the decoding process if we are given generics $G_i$ for all $i$ in some bad subset of $I$ which is not in $\mathcal A$. In light of the global well-order, any element of $M$ can be coded by the index of a condition in the proper-class-sized antichain $Z(p)$ below $p$. (This uses that the global well-order has ordertype $\Ord$.) As before, we start with $p^i_0$ being the top condition in $\P_i$. And the construction proceeds from stage $n$ to stage $n+1$ the same as before, breaking down into the same tree of cases.

Having carried out the construction, we can verify $(1)$, $(2)$, and $(3)$ the same as before. 
\end{proof}

For specific class forcings, we may get the same result assuming neither Global Choice nor Elementary Class Choice. Suppose, as an example, that each $\P_i$ is $\Add(\Ord,1)$, the class of functions from an ordinal to $\{0,1\}$, ordered by reverse inclusion. Then we can code the catastrophic real directly, as in the argument for theorem \ref{thm:Mostowski}. Indeed, this was the forcing notion Mostowski used in his original investigations. 

And once again, the exact same way as we proved \autoref{thm:EmbedFamilyOfSetsIntoCohenMultiverse} using
\autoref{thm:GeneralBlockchain}, and \autoref{cor:EmbeddingIntoWideMultiverse} using \autoref{thm:WideAmalgamability}, we can derive the following as a corollary.

\begin{corollary}
	Let $(M,\mathcal{X})$, $I$, $\mathcal{A}$, and the forcing notions $\P_i$ be as in \autoref{thm:WideClassForcing}.
	Then $(\mathcal{A},\subseteq)$ $*$-embeds into the class-generic multiverse of $(M,\mathcal{X})$ given by products of the forcing notions
	$\P_i$ (and mapping $\varnothing$ to $(M,\mathcal{X})$).
\end{corollary}

\section{Surgery and the Mutable Blockchain}

In this section we wish to revisit the results from \autoref{sec:Blockchain} and present
an alternative proof of Mostowski's \autoref{thm:Mostowski}.
While the obtained result is substantially the same, we believe that the
method we employ, \emph{surgery}, is quite powerful and can be used to answer other questions in the
area. The main difference between Mostowski's argument (and the derived blockchain construction) 
and the surgery method we are about to give is that the blockchain argument creates the required
generic objects (and thus the $*$-embeddings) completely anew, whereas with the surgery approach we will
be able to start with generic objects given in advance and modify them in such a way that they
realize the desired $*$-embedding.

We remind the reader that, in this section, we work exclusively with the version of Cohen forcing 
$\Add(\omega,1)$ whose conditions are finite binary strings. We will comment later why other
presentations of this forcing notion are not suitable.

The following definition introduces a useful way of surgically modifying one function to match another
one on a subset of its domain.

\begin{definition}
	Let $f$ and $g$ be (partial) functions defined on $\omega$ and let $A\subseteq\dom(f)$
	have the same cardinality as $\dom(g)$. Let $e\colon A\to \dom(g)$ be the unique order-preserving
	bijection.
	The \emph{graft} of $g$ onto $f$ on $A$ is the function $f\graft_A g$, defined on
	$\dom(f)$ as:
	\[
	(f\graft_A g)(x):= \begin{cases}
	f(x)& x\notin A\\
	g(e(x))& x\in A 
	\end{cases}
	\]
In other words, the graft $f\graft_A g$ is the result of replacing the values of $f$ on $A$
with the corresponding values of $g$. If $A=\dom(g)$, we omit it and just write $f\graft g$.
\end{definition}

\begin{definition}
	Let $p$ be a condition in $\Add(\omega,1)$, let $D$ be a dense open subset of this poset,
	and let $n$ be a natural number. We say that $p$ is \emph{$(n,D)$-immune} if any condition
	$q$ satisfying $q(i)=p(i)$ for all $i\geq n$ is in $D$. In other words, a condition is
	$(n,D)$-immune if any modification of it on the coordinates below $n$ results in a condition
	in $D$.
	
	Given conditions $p$ and $q$, we say that $q$ is a \emph{$D$-immunization} of $p$ if $q\leq p$
	and $q$ is $(|p|,D)$-immune.
\end{definition}

We should note that any $(n,D)$-immune condition is itself already in $D$.

\begin{proposition}
\label{prop:ImmunizationsExist}
	Let $D$ be a dense open subset of $\Add(\omega,1)$. Any condition $p\in\Add(\omega,1)$
	has a $D$-immunization.
\end{proposition}

\begin{proof}
	Fix a condition $p$. Since we are working with conditions that are binary sequences, there are only
	finitely many conditions of length $|p|$, and we enumerate them as $\seq{q_i}{i<N}$.
	We will build a descending sequence of increasingly immune conditions $p_i$ for $i\leq N$. 
	Start by letting $p_0:=p$. In general, given $p_i$, we first let $\bar{p}_{i+1}$ be an 
	extension of $p_i\graft q_i$ inside $D$ and then define $p_{i+1}:=\bar{p}_{i+1}\graft p$.
	It is then clear that $p_N$ is a $D$-immunization of $p$.
\end{proof}

The existence of immunizations is the main reason why we need to work with the binary sequence
version of Cohen forcing in this section. Working instead with conditions as finite sequences of natural numbers,
for example, it is not hard to come up with a dense open set $D$ such that no condition at all
is even $(1,D)$-immune.

The utility of immune conditions is that they admit limited amounts of surgery, while remaining
in a given dense set. The plan is then to build a generic filter from a sequence of increasingly
immune conditions with the hope that the resulting Cohen real will be able to withstand
surgery on unboundedly many coordinates while remaining generic.
As a simple example, we give the following result, which amounts to an improved version of
\autoref{prop:NonamalgamablePair}.

\begin{theorem}
\label{thm:MutuallyCohenCovert}
	Let $c$ and $d$ be mutually generic Cohen reals over $M$ and let $g\colon\omega\to 2$
	be a function (not necessarily in $M$). Then $c\graft_d g$ is a Cohen real over $M$.
\end{theorem}

\begin{proof}
	Let $D\in M$ be a dense open subset of $\Add(\omega,1)$. We wish to check that
	$c':=c\graft_d g$ meets $D$. Consider the following subset of $\Add(\omega,2)$:
	\[
	E:=\set{(p,q)}{\text{for some $k<\omega$, $p$ is $(k,D)$-immune and $q$ has $0$s on the interval $[k,|p|)$}}
	\]
	The set $E$ is dense in $\Add(\omega,2)$. This is because, given any condition
	$(p,q)$ in this poset (where we may assume that $|p|=|q|$), we can simply find a
	$D$-immunization $p'$ of $p$, using \autoref{prop:ImmunizationsExist} and pad 
	$q$ with $0$s up to the length of $p'$.
	
	Since $E$ is dense, there is a pair $(p,q)\in E$ such that $p$ and $q$ are initial
	segments of $c$ and $d$, respectively. Now consider $p\graft_q g$, where we identify
	$q$ with $q^{-1}[\{1\}]$. This is an initial segment of $c\graft_d g$, and it lies
	in $D$, since $p$ was immune up to the last $1$ of $q$ below $|p|$. Therefore
	$c\graft_d g$ meets $D$ and is truly a Cohen real over $M$.
\end{proof}

\begin{corollary}
\label{cor:EveryCohenPartOfNonamalgamablePair}
	Let $c$ be a Cohen real over $M$. Then there is another Cohen real $d$ over $M$ which does not 
	amalgamate with	$c$.
\end{corollary}

\begin{proof}
	Let $d'$ be a Cohen real mutually generic with $c$, and fix a catastrophic real $z$
	for $M$. By \autoref{thm:MutuallyCohenCovert}, the real $d:=d'\graft_c z$ is Cohen over
	$M$. But any model containing both $c$ and $d$ would be able to recover $z$
	as the bits of $d$ on the coordinates in $c$, and therefore cannot be a forcing extension of
	$M$. Thus $c$ and $d$ are not amalgamable.
\end{proof}

This corollary provides a glimpse of possible improvements to results establishing the existence
of $*$-embeddings into the generic multiverse, such as \autoref{cor:FinitePosetsEmbed}. Specifically,
it hints at a possible extension-of-embeddings phenomenon, whereby a $*$-embedding of a 
subposet into the generic multiverse could be extended to a $*$-embedding of the whole poset.
Recall that the Turing degrees enjoy this kind of property: if $Q$ is a subposet of $P$,
and moreover an initial segment, then any embedding of $Q$ into the Turing degrees extends
to an embedding of $Q$.\footnote{This result is essentially sharp. The extension-of-embeddings
problem in the Turing degrees, and various substructures such as the c.\,e.\ degrees or the $\Delta^0_2$ degrees, has a long and complex history. The interested reader might 
consult~\cite[Chapters II \& VII]{Lerman:DegreesOfUnsolvability} or the survey paper~\cite{Shore2006:DegreeStructures}.}
It seems likely that a similar result can be proved for the generic multiverse.

\begin{question}
	Does the generic multiverse exhibit an extension-of-embeddings phenomenon? That is, if $Q \subseteq P$ are posets and $Q$ embeds (or $*$-embeds) into the order structure of the generic multiverse, must it be that the embedding extends to an embedding (or $*$-embedding) of all of $P$? If not, is this true under reasonable assumptions on $P$ and $Q$?
\end{question}

In the case of $*$-embeddings, it is clear that additional requirements should be put on
$P$ and $Q$ in order to avoid trivialities. For example, if two points in $Q$ 
do not have an upper bound in $Q$ but have one in $P$, then no $*$-embedding of $Q$ can extend
to a $*$-embedding of $P$. One should therefore require, at least, that $P$ not add
any upper bounds to points in $Q$ that did not have them before.
Answering this question would also have a bearing on the decidability of the
$\forall\exists$-theory of the generic multiverse, improving on the arguments from
\autoref{cor:ExistsTheoryDecidable} and \autoref{cor:ExistsSemilatticeTheoryDecidable}.

One can also consider a wider context.
Friedman and Hathaway~\cite{FriedmanHathaway:GenericCodingWithHelp} recently used a completely different coding
method based on Hechler forcing to show that, given any pair of distinct transitive models of set theory of the same height, there is a forcing extension of one which
does not amalgamate with the other. This gives a version of \autoref{cor:EveryCohenPartOfNonamalgamablePair} in the context of the hyperuniverse.
Although Hechler forcing is not wide, and so \autoref{thm:WideAmalgamability} does not apply,
there might be a version of the blockchain argument that can be used with that forcing
and their coding method to give a variety of results analogous to the ones in this paper.
Notably, Friedman and Hathaway leave the extension-of-embeddings problem for the hyperuniverse
open, and it seems that this should serve as a significant goal, likely requiring the combination
of their and our methods.

\begin{question}
	Let $P$ and $N$ be finite sets of countable transitive models of set theory of the
	same height and suppose that no model in $N$ is a subset of a model in $P$.
	Is there a transitive model of that same height that amalgamates with each model in $P$
	and does not amalgamate with any model in $N$?
\end{question}

Before we perform surgery on a family of Cohen reals to obtain a desired pattern of amalgamability, 
we will have to prepare them slightly.
We now describe this transformation of a family of reals, called \emph{priming}.

\begin{definition}
	Let $y\colon\omega\to 2$ be a real. We say that $i$ is a \emph{good point} for $y$
	if $y(i)=1$ and $y(i+1)=0$. We write $G(y):=\set{i}{\text{$i$ is a good point for $y$}}$
	and $C(y):=\set{i+1}{i\in G(y)}$.
	
	If $x,y,z$ are reals, we write $x[z/y]:=x\graft_{C(y)}z$. Often $y$ will be clear from context
	and we will simply write $x[z]$.
	
	If $x_0,\dotsc,x_{n-1},y$ are reals, the \emph{primed} version of $\vec{x}$ with
	respect to $y$ is $\vec{x}'$, where
	\[
	x_k'(i):=
	\begin{cases}
		1& i\in G(y)\\
		0& i\in C(y)\\
		0& i\notin G(y)\cup C(y) \text{ and } |G(y)\cap i|\equiv k \pmod{n}\\
		x_k(i)& \text{otherwise}
	\end{cases}
	\]
	We say that $k$ is \emph{inactive} at $i$ if 
	$i\notin G(y)\cup C(y) \text{ and } |G(y)\cap i|\equiv k \pmod{n}$.
\end{definition}

Given reals $x_0,z_0,\dotsc,x_{n-1},z_{n-1},y$ we will quite often abuse notation and
write $\vec{x}'$ for the tuple of the primed reals $x_0',\dotsc,x_{n-1}'$, and 
$\vec{x}[\vec{z}]$ for the tuple of reals $x_0[z_0],\dotsc,x_{n-1}[z_{n-1}]$.

\begin{figure}[ht]
	\begin{tikzpicture}[scale=.4]
	\foreach \l [count=\c from 0] in {0,3,6} {
		\draw (\l,25)--(\l,0)--(\l+1,0)--(\l+1,25);
		\draw (\l+.5,-1) node[scale=1] {$x_\c'$};	
	}
	\draw (10,25)--(10,0)--(11,0)--(11,25);
	\draw (10.5,-1) node[scale=1] {$y$};
	\foreach \r [count=\c from 0] in {6,11,17,24} {
		\foreach\i in {0,3,6,10}
		{\draw (\i,\r) rectangle node[scale=.7,opacity=1] {$1$} +(1,1);
			\draw (\i,\r+1) rectangle node[scale=.7,opacity=1] {$0$} +(1,1);
		}}
	\foreach \f/\l [count=\c from 0] in {0/5,8/10,13/16,19/23} {
		\foreach \r in {\f,...,\l} {
			\draw[pattern=north east lines] (10,\r) rectangle +(1,1);
			\foreach \col in {1,2} {
				\draw[pattern=north east lines] ({3*mod(\c+\col,3)},\r) rectangle +(1,1);
			}
		}
	}
	\end{tikzpicture}
	\caption{Priming $x_0,x_1,x_2$ with respect to $y$}
	\label{fig:priming}
\end{figure}

The diagram in \autoref{fig:priming} represents the result of priming three reals $x_0,x_1,x_2$ with respect to $y$.
The blank blocks represent indices where a particular column is inactive and where $0$s have been inserted.
The operation $x_0'[z/y]$ then amounts to placing $z$ onto the indicated bits of $x_0'$
that occur just above the all-one rows. 
The picture is again divided into blocks, fit between the good 
points of $y$. As we shall see, the structure of the blocks will ensure that genericity properties
of the original reals are preserved after priming and surgery. Since the blocks are intended to
accept modification via surgery, we might describe the whole construction as the mutable blockchain.

An easy but important observation to make is that, given reals $\vec{x}$ and $y$, the primed
versions $\vec{x}'$ are never mutually generic Cohen reals, even if the original reals $\vec{x}$
were. This is because we have ensured in the priming procedure that two consecutive all-one rows 
never appear, but this would have to happen in a mutually generic family. The following lemma
shows that this is essentially the only failure of mutual genericity. If we start with mutually generic
Cohen reals, then, after priming and possibly even additional surgery, any proper subfamily
of these modified reals remains mutually generic. Moreover, this remains true even if we add some
unprimed reals into the mix.

\begin{lemma}
\label{lemma:PrimingPreservesMutualGenerics}
	Suppose that $u_0,\dotsc,u_{m-1},x_0,\dotsc,x_{n-1},y$ are mutually generic Cohen reals. Let
	$\vec{x}'$ be the primed version of $\vec{x}$ with respect to $y$.
	Let $A\subsetneq n$ and fix reals $w_0,\dotsc,w_{m-1}$ and $z_k$ for $k\in A$. Then
	$\set{u_k[w_k]}{k<m} \cup \set{x_k'[z_k]}{k\in A}$ are also mutually generic Cohen reals.
\end{lemma}


\begin{proof}
	This is much like the proof of \autoref{thm:MutuallyCohenCovert}, but expanded to
	a forcing that adds several Cohen reals and with a proper understanding of immunity
	for conditions in such a poset. For the purposes of this proof, we shall consider the posets
	$\Add(\omega,X)$ to consist of finite binary partial functions, all of whose
	nonempty columns have the same height. Moreover, if $D\subseteq \Add(\omega,X)$ is a dense
	open set, we shall say that a condition $p\in\Add(\omega,X)$ is $(n,D)$-immune if
	it remains in $D$ even after any modification of its \emph{nonempty} columns below height $n$.
	
	Let $D$ be a dense open subset of $\Add(\omega,A+1)$ and fix some $\ell\in n\setminus A$.
	Consider the following subset of $\Add(\omega,m+n+1)$:
	\begin{align*}
	E:=\{(\vec{p},\vec{q},r)\mid \exists i<\omega\colon & \text{$(\vec{p},\vec{q}\,\rest{A})$ is $(i,D)$-immune,}\\
												& \text{$r$ has only $0$s on the interval $[i,|r|)$, and}\\
												& \text{$\ell$ is inactive at every point of $[i,|r|)$}\}
	\end{align*}
	A similar argument to the one in the proof of \autoref{thm:MutuallyCohenCovert} shows that $E$
	is dense. Therefore $(\vec{u},\vec{x},y)$ meets $E$. As in that previous proof, it again
	follows by an immunity argument that $(\vec{u}\,[\vec{w}],(\vec{x}'\rest{A})[\vec{z}])$ meets $D$.
\end{proof}

Using essentially the same idea, we can extend the lemma to deal with a whole matrix of mutually generic
Cohen reals, with each row being primed with respect to a different real. We omit the proof, since
it is quite similar to the one just given.

\begin{lemma}
\label{lemma:PrimingPreservesMutualGenericsBetter}
	Suppose that $u^j_0,\dotsc,u^j_{m_j-1},x_0^j,\dotsc,x_{n_j-1}^j,y^j,v^j$ for $j<J$ are mutually generic Cohen reals.
	Let $(\vec{x}^j)'$ be the primed version of $\vec{x}^j$ with respect to $y^j$. Let
	$A^j\subsetneq n_j$ and fix reals $w^j_0,\dotsc, w^j_{m_j-1}$ and $z^j_k$ for $k\in A^j$. Then
	\[\set{(x^j_k)'[z^j_k/y^j]}{k\in A^j,j<J}\cup \set{u^j_k[w^j_k/v^j]}{k<m_j,j<J}\] are also mutually generic Cohen reals.
\end{lemma}

Before we give the main theorem of this section, let us present an easier version that we will
use as a building block in the coming proof.

\begin{theorem}
\label{thm:MutableBlockchainBaby}
	Let $n$ be a natural number and let	$A\subseteq n$. Let $y$ be a Cohen real over $M$.
	Then there are amalgamable Cohen reals $c_0,\dotsc,c_{n-1}$ over $M$
	and a real $z$ such that the family $\set{c_k[z/y]}{k\in B}$ is nonamalgamable
	if $A\subseteq B$ and mutually generic otherwise.
\end{theorem}

\begin{proof}
	Let $x_0,\dotsc,x_{n-1}$ be mutually generic Cohen reals over $M[y]$ and let $z$
	be catastrophic for $M$. Let $\set{c_k}{k\in A}$ be the primed versions of
	$\set{x_k}{k\in A}$ with respect to $y$ and let $c_k:=x_k$ if $k\notin A$.
	The reals $\set{c_i}{i<n}$ are Cohen over $M$ by \autoref{lemma:PrimingPreservesMutualGenerics},
	and they are clearly amalgamable since they all live in
	$M[x_0,\dotsc,x_{n-1},y]$. Note that, given the set $\set{c_k[z/y]}{k\in A}$,
	it is simple to decode $z$: Find those $i$ for which $c_k[z/y](i)=1$ for
	all $k\in A$ and recover $C(y)$ from that information, which then give $z$. Therefore,
	the family $\set{c_k[z/y]}{k\in B}$ is not amalgamable if $A\subseteq B$.
	On the other hand, if $A\nsubseteq B$, \autoref{lemma:PrimingPreservesMutualGenerics}
	implies that $\set{c_k[z/y]}{k\in B}$ are mutually generic Cohen reals.
\end{proof}

\begin{theorem}
\label{thm:MutableBlockchain}
	Let $n$ be a natural number. 
	Then there are amalgamable Cohen reals $c_0,\dotsc, c_{n-1}$ over $M$ and another real $y$ such 
	that, for any $\mathcal{A}\subseteq\pow(n)$ containing all singletons and closed under subsets,
	there is a real $z_{\mathcal{A}}$ such that
	$\set{c_k\graft_y z_{\mathcal{A}}}{k\in A}$ is amalgamable if and only if 
	$A\in\mathcal{A}$.
\end{theorem}

While this theorem gives a similar result as Mostowski's \autoref{thm:Mostowski}, the conclusion
is slightly stronger. In particular, we now obtain a fixed finite family of reals $c_k$ which,
by judicious use of surgery, can be modified to witness any desired pattern of amalgamability.
Furthermore, as will be apparent from the proof, the real $y$ can be chosen to be quite sparse
(in terms of asymptotic density, say), so that the ``postoperative'' reals will resemble
the ``preoperative'' ones as closely as desired.

\begin{proof}
	Let $\set{x_k^A,y^A}{k<n, A\subseteq n, |A|\geq 2}$ be mutually generic Cohen reals over $M$.
	Let $\set{c_k^A}{k\in A}$ be the primed versions of $\set{x_k^A}{k\in A}$ with respect
	to $y^A$ and let $c_k^A:=x_k^A$ for $k\notin A$. For each $k<n$ let
	$c_k:=\bigoplus_{A} c_k^A$ be the join, obtained by interleaving bits for example; we also
	let $y=\bigoplus_A C(y^A)$.
	Essentially we prepared a matrix of Cohen reals, with one row for each subset $A\subseteq n$,
	and priming the reals in the $A$-th row with indices in $A$.	We claim these $c_k$,
	the joins of the columns of this matrix, are as desired.
	
	Let us first see that each $c_k$ is Cohen over $M$.
	Let $A_0,A_1,\dotsc,A_N$ be an enumeration of subsets of $n$ of size at least $2$.
	Recall that $\set{x_k^A,y^A}{A\subseteq n, |A|\geq 2}$ are mutually generic over $M$.
	In particular, the reals $x_0^{A_0},\dotsc,x_{n-1}^{A_0},y^{A_0}$ are mutually generic
	over $M[x_k^{A_1},y^{A_1},\dotsc,x_k^{A_N},y^{A_N}]$. Passing to the primed versions,
	\autoref{lemma:PrimingPreservesMutualGenerics} tells us that any subfamily of
	$\set{c_\ell^{A_0}}{\ell<n}$ which omits at least one $c_\ell^{A_0}$ for $\ell\in A_0$
	remains mutually generic over that model. In particular, $c_k^{A_0}$ is
	Cohen over $M[x_k^{A_1},y^{A_1},\dotsc,x_k^{A_N},y^{A_N}]$.	
	Notice that	$c_k^{A_j}$ is in this model for each $j>0$. Similarly,
	$c_k^{A_1}$ is Cohen over $M[x_k^{A_2},y^{A_2},\dotsc,x_k^{A_N},y^{A_N}]$, and
	this model contains all of the reals $c_k^{A_j}$ for $j>1$. Repeating this step,
	we see that the reals $c_k^{A_0},\dotsc,c_k^{A_N}$ are mutually generic over $M$.
	Therefore their join $c_k$ is also Cohen over $M$. Furthermore, as we have said, the model 
	$M[x_k^{A_j},y^{A_j}\mid j\leq N]$ has all of the reals $c_k^{A_j}$,
	and therefore also $c_k$. Therefore all of the reals $c_k$ appear in the model
	$M[x_k^A,y^A\mid k<n, A\subseteq n, |A|\geq 2]$, and are thus amalgamable.
	
	Now fix a family $\mathcal{A}$ as in the statement of the theorem. Notice that, for a fixed
	set of indices $A$, the reals $c_k^A$ have been prepared exactly as in the proof of
	\autoref{thm:MutableBlockchainBaby}. Arguing as in that proof, there is, for each
	$A\in\pow(n)\setminus\mathcal{A}$, a real $z^A$ such that $\set{c_k^A[z^A/y^A]}{k\in B}$ 
	is nonamalgamable if
	$A\subseteq B$ and mutually generic otherwise. For $A\in\mathcal{A}$ let
	$z^A\equiv 0$. Now let 
	$z_{\mathcal{A}}:=\bigoplus_A z^A$,	so that $c_k\graft_y z_{\mathcal{A}}= \bigoplus_{A} c_k^A[z^A/y^A]$.
	
	If $A\notin\mathcal{A}$ then $\set{c_k\graft_y z_{\mathcal{A}}}{k\in A}$
	is not amalgamable, since from these reals we could compute
	$\set{c_k^A[z^A/y^A]}{k\in A}$ and we know these reals are not amalgamable.
	
	Now consider the situation when $A\in\mathcal{A}$. 
	Let us write $M^A=M[x^D_l,y^D\mid D\in\mathcal{A},l\in A]$.
	On the one hand, note that for any $B\in\mathcal{A}$ and any $k\in A$ we have
	$c^B_k[z^B/y^B]\in M^A$.
	On the other hand, applying \autoref{lemma:PrimingPreservesMutualGenericsBetter} over $M^A$ shows that, 
	since $B\setminus A\neq\varnothing$ for any $B\notin\mathcal{A}$,
	the reals in $\set{c^B_k[z^B/y^B]}{B\notin\mathcal{A},k\in A}$ are mutually generic over $M^A$.
	It follows that there is a single Cohen extension of $M^A$ containing all the reals in
	$\set{c^B_k[z^B/y^B]}{B\subseteq n,k\in A}$, and so the reals
	$c_k\graft_y z_{\mathcal{A}}$ are amalgamable over $M$.
\end{proof}

We should point out that the reals $\set{c_k\graft_y z_{\mathcal{A}}}{k\in A}$ will not 
be mutually generic, even when $A\in\mathcal{A}$. 
If that were the case then, taking into account how $c_k, y,$ and $z_{\mathcal{A}}$ were built, 
the family $\set{c_k^A[z^A/y^A]}{k\in A}$ would also have to be mutually generic.
But note that $z^A\equiv 0$, so that $c_k^A[z^A/y^A]$ is just equal to $c_k^A$,
and we have already noted that an entire primed family $\set{c_k^A}{k\in A}$ cannot be mutually
generic since it has no occurrences of two consecutive all-one rows.

A simple extension of the proof of \autoref{thm:MutableBlockchain} yields a sequence of Cohen reals 
$c_0,c_1,\dotsc$ that uniformly realizes all finite patterns of amalgamability.

\begin{theorem}
	There are amalgamable Cohen reals $c_0,c_1,\dotsc$ over $M$ and a real $y$ such that,
	for any natural number $n$ and any $\mathcal{A}\subseteq\mathcal{P}(n)$ containing all singletons
	and closed under subsets, there is a real $z_\mathcal{A}$ such that, for any $A\subseteq n$,
	the family $\set{c_k\graft_y z_\mathcal{A}}{k\in A}$ is amalgamable if and only if
	$A\in\mathcal{A}$.
\end{theorem}

It is unclear whether the mutable blockchain argument can be modified to achieve
the kind of control over the intersections of the resulting generic extensions which would be
required to obtain a $*$-embedding into the generic multiverse, or whether it can accommodate
infinite patterns $\mathcal{A}$ (compare with the modifications going from \autoref{thm:Mostowski}
to \autoref{thm:GeneralBlockchain}).

\section{More Structure in the Cohen Multiverse}
\label{sec:MoreStructure}

In this section we would like to further discuss the structure of the generic multiverse, but more
specifically restricted to those extensions which may be obtained by adding a single Cohen real.
\autoref{thm:EmbedFamilyOfSetsIntoCohenMultiverse} and \autoref{cor:FinitePosetsEmbed} show that this
structure exhibits a large degree of universality---any finite poset $*$-embeds into it, along
with certain infinite ones, and the embedding can also be made to preserve infima. We believe these
results should be seen as analogous to the many theorems about the order structure of the Turing
degrees (or even the c.~e.\ degrees), 
which exhibits a similar kind of universality (although not quite in terms of $*$-embeddings).

It turns out that there are nontrivial restrictions on the infinitary structure of
the Cohen multiverse. The second author and Venturi~\cite[Theorem~13]{Hamkins2016:UpwardClosureAndAmalgamationInGenericMultiverse} 
showed that countable increasing chains in this structure have upper bounds (a similar but not
completely analogous result holds for the whole generic multiverse; see the aforementioned paper and also \cite{FuchsHamkinsReitz2015:SetTheoreticGeology}).

\begin{theorem}[Hamkins--Venturi]
\label{thm:CohenUpwardClosure}
	Let $\mathcal{C}=\set{M[c_n]}{n<\omega}$ be a family of Cohen extensions of $M$
	such that every finite subfamily of $\mathcal{C}$ has a least upper bound in the Cohen multiverse.
	Then $\mathcal{C}$ is amalgamable in the Cohen multiverse.
\end{theorem}


On the other hand, $*$-embeddings, as we have defined them, do not capture this infinitary
behaviour very well. For example, the reader will quickly convince herself that
the inclusion of $\omega$ into $\omega+1$ is a $*$-embedding. Moreover, any
ideal on a set $X$ (such as the ideal of finite subsets) $*$-embeds into the full power set 
$\mathcal{P}(X)$, regardless of which (infinitary) upper bounds exist in the ideal.
It follows that any ideal on $\omega$ $*$-embeds into the Cohen multiverse, even though the
ideal will typically not have the finite-obstacle property that we required in our theorems.


To better explore this aspect of the structure of the Cohen multiverse, one might wish to
work with a stronger form of $*$-embedding.

\begin{definition}
	Let $f\colon P\to Q$ be a $*$-embedding between posets. We say that $f$ is a
	\emph{strong $*$-embedding} if, whenever $X\subseteq P$ does not have an upper bound 
	in $P$, its image $f[X]$ also does not have an upper bound in $Q$.
\end{definition}

It is easy to see that if $\mathcal{A}$ is a family of sets as in \autoref{thm:GeneralBlockchain}
then any $*$-embedding of $\mathcal{A}$ already is a strong $*$-embedding. Therefore
the $*$-embeddings obtained in \autoref{thm:EmbedFamilyOfSetsIntoCohenMultiverse} and
\autoref{cor:FinitePosetsEmbed} can be taken to be strong $*$-embeddings. On the other hand,
it follows from \autoref{thm:CohenUpwardClosure} that the finite-obstacle property
is a necessary condition for a family $\mathcal{A}$ to strongly
$*$-embed into the Cohen multiverse in any way similar to what we have seen.

\begin{proposition}
	Let $I\in M$ and let $\mathcal{A}\in M$ be a family of subsets of $I$, closed under subsets in $M$. Suppose that $\mathcal{A}$ strongly $*$-embeds into the Cohen multiverse via a
	$*$-embedding that preserves finite suprema.
	Then $\mathcal{A}$ is defined by finite obstacles in $M$.
\end{proposition}

Note that the $*$-embeddings we constructed in \autoref{sec:Blockchain} (and any embedding which,
like those, is based on arrangements of product generics) were strong $*$-embeddings and
preserved finite suprema.

\begin{proof}
	Let $f$ be the strong $*$-embedding of $\mathcal{A}$ into the Cohen multiverse.
	Fix a set $B\in M$ all of whose finite subsets are in $\mathcal{A}$. We aim to show that
	$B$ is also in $\mathcal{A}$. 
	Since $M$ is countable, there are, from the point of view of $V$, only countably many finite subsets of $B$. 
	Let $\mathcal{C}$ be the range
	of $f$ restricted to the finite subsets of $B$. Note that $\mathcal{A}$
	is closed under unions of finite subsets of $B$, so, since $f$ preserves suprema,
	$\mathcal{C}$ is closed under suprema. The family $\mathcal{C}$ is countable, since
	$M$ is countable, therefore it follows from \autoref{thm:CohenUpwardClosure} that
	$\mathcal{C}$ is amalgamable in the Cohen multiverse. By the definition of a
	strong $*$-embedding we can thus conclude that $B\in\mathcal{A}$. So any subset of $I$ not in $\mathcal{A}$ must contain a finite set not in $\mathcal{A}$, and thus $\mathcal{A}$ is defined by finite obstacles. The set of obstacles can be taken as the $\subset$-minimal subsets of $I$ not in $\mathcal{A}$, and this set is in $M$.
\end{proof}

Continuing with the comparison of the Cohen multiverse and Turing degrees, we now turn
to the existence of suprema and infima. First of all, it is an easy consequence of the intermediate
model theorem that we can find a supremum of finitely many models in the generic multiverse as
long as they are amalgamable.

\begin{proposition}
	\label{prop:AmalgamableImpliesSup}
	A finite subset of the generic multiverse has a supremum if and only if it is amalgamable.
\end{proposition}

\begin{proof}
	Consider two models $M_0$ and $M_1$ in the generic multiverse. It follows from the
	\DDG (see \autoref{sec:intro}) that both of these two models are forcing extensions of
	a common ground model. We may assume without loss of generality that this common ground
	model is just $M$ and write $M_0=M[G_0]$ and $M_1=M[G_1]$ for some
	$M$-generic $G_0$ and $G_1$. We may also assume that $G_0$ and $G_1$ are sets
	of ordinals. Now suppose that $M_0$ and $M_1$ are amalgamable with upper bound
	$M_2$ and consider the model $M(G_0,G_1)=\bigcup_{\alpha<\Ord^M} L(V_\alpha^M,G_0,G_1)^{M_2}$.
	This is an inner model of $M_2$ and it satisfies Choice since $G_0$ and $G_1$ are
	sets of ordinals and therefore easily well-orderable. It then follows from the intermediate
	model theorem that $M(G_0,G_1)$ is a forcing extension of $M$, and it is also clearly
	the least upper bound for $M[G_0]$ and $M[G_1]$. 
\end{proof}

In contrast, it is known that countable increasing sequences of Turing degrees 
never have a supremum and that not all pairs of Turing degrees have an infimum .
These facts are usually established by showing the existence of \emph{exact pairs} of Turing degrees,
and the analogous argument will show that the same holds in the Cohen multiverse.

%

\begin{definition}
	Let $\mathcal{E}$ be a family of forcing extensions of $M$. We say that two distinct forcing extensions
	$M[G]$ and $M[H]$ form an \emph{exact pair} for $\mathcal{E}$ if each of them is an
	upper bound for $\mathcal{E}$, and any model in the generic multiverse below both $M[G]$ and 
	$M[H]$ is below some model in $\mathcal{E}$.
\end{definition}

\begin{theorem}
\label{thm:ExactPair}
	Suppose that $M[c_0]\subseteq M[c_1]\subseteq\dots$ is a countable tower
	of Cohen extensions of $M$. Then there are Cohen extensions $M[d_0]$ and $M[d_1]$ forming an 
	exact pair for this tower.
\end{theorem}

\begin{proof}
	First we replace the Cohen reals $c_n$ with a different family $c_n'$ giving the same
	models, but with the additional property that any finitely many of the $c_n'$ are mutually generic.
	We begin by letting $c_0'=c_0$.
	By the intermediate model theorem, $M[c_1]$ is a forcing extension of $M[c_0]$ by a quotient
	of the forcing $\Add(\omega,1)$. But any such quotient is itself equivalent to $\Add(\omega,1)$
	and so there is a Cohen real $c_1'$ over $M[c_0]=M[c_0']$ such that $M[c_1]=M[c_0',c_1']$.
	Continuing in this way we can find all of the reals $c_n'$ as required.
	
	We will build the exact pair by filling in two $\omega$-by-$\omega$ matrices with $0$s and $1$s
	in $\omega$ many steps. Throughout the construction we will maintain the requirement
	that, in each matrix, finitely many of the leftmost columns have been completely filled
	and the $n$-th such column differs at most finitely much from the corresponding real $c_n'$,
	and that beyond these columns the matrix is empty.
	In particular, at any stage of the construction, the matrices will essentially consist of
	finitely many mutually generic Cohen reals.
	
	In our construction we will run through a fixed enumeration of all dense subsets of
	$\Add(\omega,\omega)$, as well as all pairs of $\Add(\omega,\omega)$-names for
	reals in $M$. So suppose that we are at some stage of the construction and we are
	handed a dense set $D$ and a pair of names $\sigma$ and $\tau$. Suppose
	that, in each matrix, the first $N$ many columns have been filled. Since these
	columns are mutually generic, they meet the projection of $D$ onto the first
	$N$ many coordinates. It follows that we can extend each matrix by finitely many
	bits to meet the whole dense set $D$. Now consider the two names $\sigma$ and
	$\tau$, or rather, consider the two $\Add(\omega,\omega\setminus N)$-names
	$\sigma'$ and $\tau'$ obtained by partially evaluating $\sigma$ and $\tau$
	by the first $N$ columns of the respective matrices. Let $p$ and $q$
	be the finite parts of the two matrices, seen as conditions in 
	$\Add(\omega,\omega\setminus N)$. If there is a number $k\in\omega$
	such that there are extensions $p'\leq p$ and $q'\leq q$ which
	force $\check{k}\in \sigma'$ and $\check{k}\notin\tau'$ (or vice versa),
	respectively, we extend our matrices by these conditions; otherwise we do nothing.
	Note that if such a $k$ does not exist then $p$ and $q$ already decide
	all formulas $\check{k}\in\sigma'$ and $\check{k}\in\tau'$ and decide them
	in the same way. Finally, we consider each partially filled column in our matrices and complete
	it to (almost) match the corresponding real $c_n'$. This finishes this step of the construction.
	
	At the end of the construction let $d_0$ and $d_1$ be the reals represented by
	the two matrices we built. It is clear that both of these reals are Cohen over $M$,
	since we met all the relevant dense sets in $M$, and it is also clear that
	$M[d_0]$ and $M[d_1]$ are upper bounds for the tower we started with, since
	$d_0$ and $d_1$ code all the reals $c_n'$ (up to finite modification).
	It thus only remains to show that no models can be fit between the tower and the new
	pair of extensions. By the intermediate model theorem, any putative such model would
	have to itself be a Cohen extension of $M$, so it suffices to show that any real common to both 
	$M[d_0]$ and $M[d_1]$ already appears at some point in the tower. So let $x$ be such a real
	and fix two names $\sigma$ and $\tau$ from $M$ such that $\sigma^{d_0}=\tau^{d_1}=x$.
	This pair of names was considered at some point in the construction, and we attempted
	at that time to ensure that $\sigma^{d_0}$ and $\tau^{d_1}$ would differ at some
	$k$. By our hypothesis there was no $k$ like this, so, as we argued above, the set
	$x=\sigma^{d_0}=\tau^{d_1}$ appeared already in the extension $M[c_N]$.
\end{proof}

\begin{corollary}
	No countable increasing tower of Cohen extensions has a least upper bound in the generic 
	multiverse.
\end{corollary}

\begin{proof}
	By the preceding theorem any such tower admits an exact pair $M[d_0], M[d_1]$, and by definition
	there is no upper bound for the tower that would lie below both of these two models. Therefore
	there can be no least upper bound for the tower.
\end{proof}

A weaker version of \autoref{thm:ExactPair} was first obtained by Balcar and 
Hájek~\cite{BalcarHajek1978:SequencesOfDegrees}; they proved that some tower of Cohen extensions
admits an exact pair (see also the elaboration by 
Truss~\cite{Truss1978:IncreasingSequencesOfConstructibilityDegrees}).
Our results generalize this further to show that any countable tower of Cohen extensions admits an 
exact pair and, as we are about to see, that any upper bound for a tower may be extended to an exact 
pair.

\begin{theorem}
	Suppose $M[c_0]\subseteq M[c_1]\subseteq\dots$ is a countable tower of Cohen extensions
	of $M$ with upper bound $M[d_0]$. Then there is a Cohen real $d_1$ over $M$ such that
	$M[d_0]$ and $M[d_1]$ form an exact pair for this tower.
\end{theorem}

\begin{proof}
	We proceed as in the proof of \autoref{thm:ExactPair}, replacing the reals $c_n$ with the 
	finitely mutually generic $c_n'$, and building the real $d_1$ as an $\omega$-by-$\omega$
	matrix whose $n$-th column is almost equal to $c_n'$. Arguing as before, at each step
	we first meet a dense subset of $\Add(\omega,\omega)$ from $M$ and then consider a pair
	of names $\sigma$ and $\tau$ for reals. The only case in which we should act is if
	$d_1$ (or the part of it that we have constructed thus far) does not decide the value
	of $\tau$. In that case there is some $k\in\omega$ such that our approximation of $d_1$
	does not decide $k\in \tau$, and we extend it to make $\tau$ and $\sigma^{d_0}$ differ at 
	$k$ (note that $d_0$ is fully generic, so $\sigma^{d_0}$ is a fully fledged 
	real). Afterwards we fill in the nonempty columns of $d_1$ using the given Cohen reals $c_n'$.
	
	The proof is finished exactly the same way as before. The constructed real is clearly Cohen over
	$M$, and any real that can be written as $\sigma^{d_0}=\tau^{d_1}$ must have already
	been fully decided at some initial stage of the construction and must thus appear in some
	extension $M[c_N]$.
\end{proof}

\begin{corollary}
	For any Cohen extension $M[c]$ there is another Cohen extension $M[d]$ such that these two
	models do not have a greatest lower bound in the generic multiverse.
\end{corollary}

\begin{proof}
	Let $c$ be a Cohen real over $M$, seen as an $\omega$-by-$\omega$ matrix, and let
	$c_n$ be the real consisting of the first $(n+1)$ many columns of $c$. Then
	all of the reals $c_n$ are Cohen over $M$ and they form a tower
	$M\subseteq M[c_0]\subseteq M[c_1]\subseteq\dots$ with upper bound $M[c]$. 
	By the preceding theorem
	there is a Cohen real $d$ such that $M[c]$ and $M[d]$ form an exact pair for this tower.
	But these two models cannot have a greatest lower bound, since each model in the tower
	is a lower bound for them, but, by definition, no lower bound exists above the tower.
\end{proof}

Note that if $d_0$ and $d_1$ are an exact pair for the tower
$M[c_0]\subseteq M[c_1]\subseteq M[c_2]\subseteq\dots$,
it will not in general be the case that $M[d_0]\cap M[d_1]=\bigcup_n M[c_n]$,
even though this is true at the level of reals.
For example $\set{x}{\exists n\colon x=^* c_n}$, the set of all finite modifications
of the $c_n$, is in both $M[d_0]$ and $M[d_1]$ as constructed in the proof of
\autoref{thm:ExactPair}, but is clearly not in any $M[c_n]$.

Finally, we wish to examine the relationship between the existence of suprema and infima
of models in the multiverse and mutual genericity. Two mutually generic extensions clearly
have a supremum. On the other hand, Solovay showed that the intersection of two mutually
generic extensions is exactly the ground model, and therefore the ground model is also
the infimum of these two extensions in the multiverse. As we will see, neither of these two
implications can be reversed and, in fact, the existence of a supremum does not imply the
existence of an infimum or vice versa.

\begin{theorem}
	There is a pair of models in the Cohen multiverse that have an infimum but do not have a supremum,
	and another pair of models that have a supremum but do not have an infimum. In particular, neither
	the existence of a supremum nor the existence of an infimum imply that the two models are mutually generic (over 
	the infimum, in the second case).
\end{theorem}

\begin{proof}
	For the first part, consider two extensions $M[c_0], M[c_1]$ given by 
	\autoref{thm:GeneralBlockchain} for the family $\mathcal{A}=\{\{0\},\{1\}\}$. In particular,
	$M[c_0]$ and $M[c_1]$ do not amalgamate, so they do not have a supremum, 
	and their intersection is $M$, which implies that $M$ is also their greatest lower bound in 
	the multiverse.
	
	For the second part we use a result due to 	
	Truss~\cite{Truss1978:IncreasingSequencesOfConstructibilityDegrees} which states that
	in any Cohen extension $M[c]$ there are two Cohen reals $d,e$ over $M$ such that
	$M[d]$ and $M[e]$ do not have an infimum. His construction builds an exact pair
	over a tower much like we did in \autoref{thm:ExactPair}, but instead of building the pair
	by induction and ensuring that the two generics almost agree on each column, Truss
	uses a generic sequence of finite modifications on the columns of $c$. 
	In the end we are left with
	two models $M[d]$ and $M[e]$ with no infimum, but they are amalgamable and therefore
	have a supremum, by \autoref{prop:AmalgamableImpliesSup}.
\end{proof}

We can strengthen the preceding result a bit to show that not even the existence of both a supremum and
an infimum suffices for mutual genericity.

\begin{theorem}
	There are Cohen extensions $M[c]$ and $M[d]$ which have both a supremum and an
	infimum in the full generic multiverse, but which are not mutually generic over their infimum.
\end{theorem}

\begin{proof}
	Let $M[G]$ be an extension of $M$ by the forcing to collapse $(2^\omega)^+$ to be countable.
	Inspecting the proof of \autoref{thm:GeneralBlockchain}, it is clear that that whole
	construction may be carried out in $M[G]$ (with $G$ taking the role of the catastrophic real) 
	to produce Cohen reals $c,d\in M[G]$ over $M$
	such that any extension of $M$ containing both $c$ and $d$ also contains $G$ and that
	any real contained in $M[c]\cap M[d]$ is already contained in $M$.\footnote{Some care is 
	needed to ensure this last requirement. Instead of going through all pairs of Cohen names
	for subsets of $M$, as in the proof of \autoref{thm:GeneralBlockchain}, we just go
	through pairs of names for reals, and there are countably many of those in $M[G]$.}

	It follows that $M$ is the infimum of $M[c]$ and $M[d]$. This is because any other
	lower bound above $M$ would have to be a forcing extension of $M$ by a subforcing of
	Cohen forcing, by the intermediate model theorem, and all such extensions are generated
	by a real, but we assumed that $M[c]$ and $M[d]$ only have the reals of $M$ in common.
	On the other hand, $M[c]$ and $M[d]$ clearly have $M[G]$ as their supremum.
	
	However, $M[c]$ and $M[d]$ cannot be mutually generic. If they were, then their supremum
	would also be a Cohen extension of $M$, which it clearly is not.
\end{proof}

The Cohen extensions in the preceding theorem were amalgamable in the full multiverse but not in
the Cohen multiverse. It is less clear whether the conclusion of the theorem still holds if we
require the supremum and infimum to exist in the Cohen multiverse, but we expect that it does.

\begin{question}
	Let $M[c]$ and $M[d]$ be Cohen extensions of $M$ with infimum $M$ and a supremum in
	the Cohen multiverse. Are $M[c]$ and $M[d]$ mutually generic extensions of $M$?
\end{question}

Our construction of an exact pair of Cohen reals can be carried out in a sufficiently large
collapse extension of $M$, and Truss showed that just a Cohen extension suffices. Both of these
imply that exact pairs can be amalgamable. On the other hand, it is not clear whether the
construction can be combined in some way with the blockchain construction to ensure nonamalgamability.

\begin{question}
	Can an exact pair be nonamalgamable?
\end{question}

Finally, acknowledging the utility of exact pairs of Cohen reals, we must ask whether their
existence is a peculiar fact about Cohen forcing, or whether they can be constructed in other
multiverses as well.

\begin{question}
	Do forcing notions beyond Cohen forcing also exhibit exact pairs?
\end{question}


\bibliographystyle{amsalphafixed}
\bibliography{Bibliography}

\newcommand{\etalchar}[1]{$^{#1}$}
\providecommand{\bysame}{\leavevmode\hbox to3em{\hrulefill}\thinspace}
\providecommand{\MR}{\relax\ifhmode\unskip\space\fi MR }
\providecommand{\MRhref}[2]{%
  \href{http://www.ams.org/mathscinet-getitem?mr=#1}{#2}
}
\providecommand{\href}[2]{#2}
\begin{thebibliography}{HKL{\etalchar{+}}16}

\bibitem[AF13]{ArrigoniFriedman2013:HyperuniverseProgram}
Tatiana Arrigoni and Sy-David Friedman, \emph{{The Hyperuniverse Program}},
  Bull. Symb. Log. \textbf{19} (2013), no.~1, 77--96,
  \href{https://doi.org/10.2178/bsl.1901030}{\textsc{doi}:~10.2178/BSL.1901030}.

\bibitem[AFHT18]{AntosFriedmanHonzikTernullo2018:HyperuniverseProjectAndMaximality}
Carolin Antos, Sy-David Friedman, Radek Honz\'ik, and Claudio Ternullo (eds.),
  \emph{{The Hyperuniverse Project and Maximality}}, Birkh\"{a}user, Basel,
  2018,
  \href{https://doi.org/10.1007/978-3-319-62935-3}{\textsc{doi}:~10.1007/978-3-319-62935-3}.

\bibitem[BH78]{BalcarHajek1978:SequencesOfDegrees}
Bohuslav Balcar and Petr H\'ajek, \emph{{On Sequences of Degrees of
  Constructibility}}, Z. Math. Logik Grundlag. Math. \textbf{24} (1978), no.~4,
  291--296,
  \href{https://doi.org/10.1002/MALQ.19780241903}{\textsc{doi}:~10.1002/MALQ.19780241903}.

\bibitem[Buk73]{Bukovsky1973:CharacterizationOfGenericExtensions}
Lev Bukovsk\'y, \emph{{Characterization of Generic Extensions of Models of Set
  Theory}}, Fundam. Math. \textbf{83} (1973), no.~1, 35--46, available at
  \url{http://eudml.org/doc/214703}.

\bibitem[FH]{FriedmanHathaway:GenericCodingWithHelp}
Sy~David Friedman and Dan Hathaway, \emph{{Generic Coding with Help and
  Amalgamation Failure}},
  \href{https://arxiv.org/abs/1808.10304}{arXiv:~1808.10304 [math.LO]}.

\bibitem[FHR15]{FuchsHamkinsReitz2015:SetTheoreticGeology}
Gunter Fuchs, Joel~David Hamkins, and Jonas Reitz, \emph{{Set-Theoretic
  Geology}}, Ann. Pure Appl. Logic \textbf{166} (2015), no.~4, 464--501,
  \href{https://doi.org/10.1016/J.APAL.2014.11.004}{\textsc{doi}:~10.1016/J.APAL.2014.11.004},
  \href{https://arxiv.org/abs/1107.4776}{arXiv:~1107.4776 [math.LO]}, comments
  and
  discussion:~\href{http://jdh.hamkins.org/set-theoreticgeology/}{jdh.hamkins.org}.

\bibitem[Fri00]{friedman:book}
Sy~David Friedman, \emph{{Fine Structure and Class Forcing}}, De Gruyter Series
  in Logic and Its Applications, De Gruyter, Berlin, 2000,
  \href{https://doi.org/10.1515/9783110809114}{\textsc{doi}:~10.1515/9783110809114}.

\bibitem[Fri12]{Friedman2012:HyperuniverseTutorial}
\bysame, \emph{{The Hyperuniverse}}, tutorial at the University of M\"unster,
  2012, \url{http://www.logic.univie.ac.at/~sdf/papers/muenster.2012.pdf}.

\bibitem[Fuj12]{fujimoto2012}
Kentaro Fujimoto, \emph{{Classes and Truths in Set Theory}}, Ann. Pure Appl.
  Logic \textbf{163} (2012), no.~11, 1484--1523,
  \href{https://doi.org/10.1016/J.APAL.2011.12.006}{\textsc{doi}:~10.1016/J.APAL.2011.12.006}.

\bibitem[GH]{gitman-hamkins2015}
Victoria Gitman and Joel~David Hamkins, \emph{{{K}elley--{M}orse Set Theory and
  Choice Principles for Classes}}, unpublished.

\bibitem[GH17]{gitman-hamkins2016}
\bysame, \emph{{Open Determinacy for Class Games}}, Foundations of Mathematics:
  Logic at Harvard. Essays in Honor of W. Hugh Woodin's 60th Birthday
  (Andrés~E. Caicedo, James Cummings, Peter Koellner, and Paul~B. Larson,
  eds.), Contemp. Math., vol. 690, Amer. Math. Soc., Providence, RI, 2017,
  pp.~121--143,
  \href{https://doi.org/10.1090/CONM/690}{\textsc{doi}:~10.1090/CONM/690},
  \href{https://arxiv.org/abs/1509.01099}{arXiv:~1509.01099 [math.LO]},
  comments and
  discussion:~\href{http://jdh.hamkins.org/open-determinacy-for-class-games/}{jdh.hamkins.org}.

\bibitem[GHJ16]{gitman-hamkins-johnstone2016}
Victoria Gitman, Joel~David Hamkins, and Thomas~A. Johnstone, \emph{{What is
  the Theory {ZFC} without Power Set?}}, MLQ Math. Log. Q. \textbf{62} (2016),
  no.~4--5, 391--406,
  \href{https://doi.org/10.1002/malq.201500019}{\textsc{doi}:~10.1002/MALQ.201500019},
  \href{https://arxiv.org/abs/1110.2430}{arXiv:~1110.2430 [math.LO]}, comments
  and
  discussion:~\href{http://jdh.hamkins.org/what-is-the-theory-zfc-without-power-set/}{jdh.hamkins.org}.

\bibitem[Git14]{Gitman2014.webpost:Kelley-Morse-set-theory-and-choice-principles-for-classes}
Victoria Gitman, \emph{{Kelley-Morse Set Theory and Choice Principles for
  Classes}}, blog post with slides, 2014,
  \url{https://victoriagitman.github.io/talks/2014/12/31/kelley-morse-set-theory-and-choice-principles-for-classes.html}.

\bibitem[Ham15]{Hamkins2015.MO222602:Is-it-consistent-with-ZFC-that-no-nontrivial-forcing-has-automatic-mutual-genericity?}
Joel~David Hamkins, \emph{{Is it consistent with ZFC that no nontrivial forcing
  notion has automatic mutual genericity?}}, MathOverflow question, 2015,
  \href{https://mathoverflow.net/q/222602}{MathOverflow:~222602}.

\bibitem[Ham16]{Hamkins2016:UpwardClosureAndAmalgamationInGenericMultiverse}
\bysame, \emph{{Upward Closure and Amalgamation in the Generic Multiverse of a
  Countable Model of Set Theory}}, RIMS {K\=oky\=uroku} \textbf{1988 Recent
  Developments in Axiomatic Set Theory} (2016), 17--31,
  \href{http://hdl.handle.net/2433/224551}{\textsc{hdl}:~2433/224551},
  \href{https://arxiv.org/abs/1511.01074}{arXiv:~1511.01074 [math.LO]},
  comments and
  discussion:~\href{http://jdh.hamkins.org/upward-closure-and-amalgamation-in-the-generic-multiverse-of-a-countable-model-of-set-theory/}{jdh.hamkins.org}.

\bibitem[HKL{\etalchar{+}}16]{HKLNS2016}
Peter Holy, Regula Krapf, Philipp L\"ucke, Ana Njegomir, and Philipp Schlicht,
  \emph{{Class Forcing, the Forcing Theorem and {B}oolean Completions}}, J.
  Symb. Log. \textbf{81} (2016), no.~4, 1500--1530,
  \href{https://doi.org/10.1017/JSL.2016.4}{\textsc{doi}:~10.1017/JSL.2016.4},
  \href{https://arxiv.org/abs/1710.10820}{arXiv:~1710.10820 [math.LO]}.

\bibitem[HKS18]{HKS2018}
Peter Holy, Regula Krapf, and Philipp Schlicht, \emph{{Characterizations of
  Pretameness and the {O}rd-cc}}, Ann. Pure Appl. Logic \textbf{169} (2018),
  no.~8, 775--802,
  \href{https://doi.org/10.1016/J.APAL.2018.04.002}{\textsc{doi}:~10.1016/J.APAL.2018.04.002},
  \href{https://arxiv.org/abs/1710.10825}{arXiv:~1710.10825 [math.LO]}.

\bibitem[HL08]{HamkinsLoewe2008:TheModalLogicOfForcing}
Joel~David Hamkins and Benedikt L\"owe, \emph{{The Modal Logic of Forcing}},
  Trans. Amer. Math. Soc. \textbf{360} (2008), no.~4, 1793--1817,
  \href{https://doi.org/10.1090/S0002-9947-07-04297-3}{\textsc{doi}:~10.1090/S0002-9947-07-04297-3},
  \href{https://arxiv.org/abs/math/0509616}{arXiv:~MATH/0509616 [math.LO]},
  comments and
  discussion:~\href{http://jdh.hamkins.org/themodallogicofforcing/}{jdh.hamkins.org}.

\bibitem[HL13]{HamkinsLoewe2013:MovingUpAndDownInTheGenericMultiverse}
\bysame, \emph{{Moving Up and Down in the Generic Multiverse}}, Logic and Its
  Applications, Lecture Notes in Comput. Sci., vol. 7750, Springer, Heidelberg,
  2013, pp.~139--147,
  \href{https://doi.org/10.1007/978-3-642-36039-8_13}{\textsc{doi}:~10.1007/978-3-642-36039-8\_13},
  \href{https://arxiv.org/abs/1208.5061}{arXiv:~1208.5061 [math.LO]}, comments
  and
  discussion:~\href{http://jdh.hamkins.org/up-and-down-in-the-generic-multiverse/}{jdh.hamkins.org}.

\bibitem[Jec03]{Jech2002:SetTheory}
Thomas Jech, \emph{{Set Theory}}, Springer Monographs in Mathematics, Springer,
  Berlin/Heidelberg, 2003,
  \href{https://doi.org/10.1007/3-540-44761-X}{\textsc{doi}:~10.1007/3-540-44761-X}.

\bibitem[Jen70]{Jensen1970:DefinableSetsOfMinimalDegree}
Ronald Jensen, \emph{{Definable Sets of Minimal Degree}}, {Mathematical Logic
  and Foundations of Set Theory (Proceedings of an International Colloquium
  Held Under the Auspices of The Israel Academy of Sciences and Humanities,
  11--14 November 1968, Jerusalem)}, North-Holland, Amsterdam, 1970,
  pp.~122--128,
  \href{https://doi.org/10.1016/S0049-237X(08)71934-7}{\textsc{doi}:~10.1016/S0049-237X(08)71934-7}.

\bibitem[Ler17]{Lerman:DegreesOfUnsolvability}
Manuel Lerman, \emph{{Degrees of Unsolvability: Local and Global Theory}},
  Perspectives in Logic, Cambridge University Press, Cambridge, 2017,
  \href{https://doi.org/10.1017/9781316717059}{\textsc{doi}:~10.1017/9781316717059}.

\bibitem[Mos76]{Mostowski1976:ModelsOfBGAxioms}
Andrzej Mostowski, \emph{{A Remark on Models of the {G}\"{o}del-{B}ernays
  Axioms for Set Theory}}, Sets and Classes: On the Work by {P}aul {B}ernays
  (Gert~H.\ M\"{u}ller, ed.), Studies in Logic and the Foundations of Math.,
  vol.~84, North-Holland, Amsterdam, 1976, pp.~325--340,
  \href{https://doi.org/10.1016/S0049-237X(09)70288-5}{\textsc{doi}:~10.1016/S0049-237X(09)70288-5}.

\bibitem[Rei07]{Reitz2007:GroundAxiom}
Jonas Reitz, \emph{{The Ground Axiom}}, J. Symb. Log. \textbf{72} (2007),
  no.~4, 1299--1317,
  \href{https://doi.org/10.2178/JSL/1203350787}{\textsc{doi}:~10.2178/JSL/1203350787},
  \href{https://arxiv.org/abs/math/0609064}{arXiv:~MATH/0609064 [math.LO]}.

\bibitem[Sho06]{Shore2006:DegreeStructures}
Richard~A. Shore, \emph{{Degree Structures: Local and Global Investigations}},
  Bull. Symb. Log. \textbf{12} (2006), no.~3, 369--389,
  \href{https://doi.org/10.2178/bsl/1154698739}{\textsc{doi}:~10.2178/BSL/1154698739},
  also available at \url{http://projecteuclid.org/euclid.bsl/1154698739}.

\bibitem[Sta84]{stanley1984}
Maurice~Collins Stanley, \emph{{A Unique Generic Real}}, {PhD} thesis,
  University of California, Berkeley, 1984.

\bibitem[Tru78]{Truss1978:IncreasingSequencesOfConstructibilityDegrees}
John~K. Truss, \emph{{A Note on Increasing Sequences of Constructibility
  Degrees}}, Higher Set Theory (Proceedings, {O}berwolfach, Germany April
  13--23, 1977) (Gert~H.\ M\"{u}ller and Dana~Stewart Scott, eds.), Lecture
  Notes in Math., vol. 669, Springer, Berlin, 1978, pp.~473--476,
  \href{https://doi.org/10.1007/BFB0103096}{\textsc{doi}:~10.1007/BFB0103096}.

\bibitem[Usu17]{Usuba2017:DownwardDirectedGrounds}
Toshimichi Usuba, \emph{{The Downward Directed Grounds Hypothesis and Very
  Large Cardinals}}, J. Math. Log. \textbf{17} (2017), no.~2, 1750009,
  \href{https://doi.org/10.1142/S021906131750009X}{\textsc{doi}:~10.1142/S021906131750009X},
  \href{https://arxiv.org/abs/1707.05132}{arXiv:~1707.05132 [math.LO]}.

\bibitem[Wil18]{williams-diss}
Kameryn~J. Williams, \emph{{The Structure of Models of Second-Order Set
  Theories}}, {PhD} thesis, The Graduate Center of the City University of New
  York, 2018,
  \href{https://academicworks.cuny.edu/gc_etds/2678/}{academicworks.cuny.edu/gc\_etds/2678/}.

\bibitem[Woo11]{Woodin2011:CHGenericMultiverseOmegaConjecture}
William~Hugh Woodin, \emph{{The Continuum Hypothesis, the Generic-Multiverse of
  Sets, and the {$\Omega$} Conjecture}}, Set Theory, Arithmetic, and
  Foundations of Mathematics: Theorems, Philosophies (Juliette Kennedy and
  Roman Kossak, eds.), Lect. Notes Log., vol.~36, Assoc. Symbol. Logic, La
  Jolla, CA, 2011, pp.~13--42.

\bibitem[Zar96]{zarach1996}
Andrzej~M. Zarach, \emph{{Replacement {$\nrightarrow$} Collection}}, G\"odel
  96: Logical Foundations of Mathematics, Computer Science, and Physics -- Kurt
  G\"odel's Legacy (Petr H\'ajek, ed.), Lect. Notes Log., vol.~6, Springer,
  Berlin/Heidelberg, 1996, pp.~307--322,
  \href{https://doi.org/10.1007/978-3-662-21963-8_22}{\textsc{doi}:~10.1007/978-3-662-21963-8\_22}.

\end{thebibliography}
\end{document}